\documentclass[10pt,a4paper,reqno]{amsproc}
\usepackage{longtable,cite}
\usepackage[a4paper, top=1in, bottom=1in, left=0.8in, right=0.8in]{geometry}
\usepackage{lscape}
\usepackage{multirow}
\usepackage[pdftex,bookmarks=true]{hyperref}
\newtheorem{theorem}{Theorem}[section]
\newtheorem{lemma}[theorem]{Lemma}

\newtheorem{proposition}[theorem]{Proposition}

\theoremstyle{definition}
\newtheorem{definition}[theorem]{Definition}
\newtheorem{example}[theorem]{Example}

\theoremstyle{remark}
\newtheorem{remark}[theorem]{Remark}
\numberwithin{equation}{section}
\usepackage{placeins}
\usepackage{hyperref}
\usepackage{tikz}
\usetikzlibrary{shapes, arrows}
\usetikzlibrary{chains,arrows.meta}
\begin{document}
		\title [Theory and Applications of Convolution-Based short time offset linear canonical transform
	]{Theory and Applications of Convolution-Based short time offset linear canonical transform} 
	\author{Gita Rani Mahato}
     \author{Manab Kundu}
	\author{Alit Elsa Xavier}
	
	\address{Department of Mathematics, SRM University AP, Amaravati-522240, India}
	
	\email{\hfill \break
		gitamahato1158$@$gmail.com (Gita Rani Mahato),
		\hfill \break manabiitism17$@$gmail.com (Manab Kundu-Corresponding author)\hfill \break alitelsaxavier@gmail.com(Alit Elsa Xavier)}
	
	\thanks{Corresponding author: Manab Kundu}
	
	\date{}
	
	\begin{abstract}
   In this paper, we introduce a convolution based short-time offset linear canonical transform (STOLCT) and investigate its fundamental mathematical properties. Specifically, we establish its continuity, orthogonality relations, inversion formulas, range theorem, and convolution theorem. We further explore several important
applications of STOLCT, including the Poisson summation formula, the Paley–Wiener criterion, and a sampling theorem. In addition, numerical simulations and graphical analyses are presented to compare signal reconstruction performance under different scenarios. A comparative study between STOLCT and STLCT is conducted with respect to their reconstruction formulas, demonstrating the effectiveness and potential advantages of the proposed transform.
	\end{abstract}
	\maketitle
	\section{Introduction} 
The Fourier transform is one of the important analytical tool in science and engineering and is widely used in signal processing, communications, and image analysis\cite{eigen, wolf}. It represents a signal in the frequency domain and reveals its overall spectral characteristics. However, the classical Fourier transform assumes that the signal is stationary, meaning that its frequency content does not vary with time. This assumption restricts its effectiveness when dealing with non-stationary signals whose spectral properties evolve over time. Moreover, because the Fourier transform uses a global kernel, it cannot capture local frequency variations within a signal.

To overcome this limitation, the windowed Fourier transform (WFT), also called the short-time Fourier transform (STFT), was introduced. By applying a localized window function, the WFT provides a joint time–frequency representation, enabling the analysis of signals with time-varying spectral content. This idea forms the basis of many modern time–frequency analysis methods\cite{Gröchenig}. Although windowed Fourier analysis offers time localization, it still inherits structural limitations from the classical Fourier transform when analyzing chirp-like or quadratic-phase signals. The Fourier kernel 
$e^{-i\xi t}$ represents pure harmonic oscillations and is therefore well suited for stationary sinusoidal components. In contrast, many real-world signals exhibit time-varying frequency behaviour that cannot be efficiently modeled within this harmonic framework.

In recent years, to address the limitations of the Fourier transform in handling non-stationary signals, several generalized transforms have been developed such as the fractional Fourier transform, the linear canonical transform, quadratic phase Fourier transform, and the offset linear canonical transform. These transforms replace the standard Fourier kernel with more flexible kernels that incorporate additional parameters, such as quadratic-phase and offset terms\cite{Shah, Stanković, Kou}. As a result, they provide a more natural representation of chirp-like and modulated signals while preserving the linear integral operator  of the Fourier transform. This generalization significantly enhances analytical flexibility and broadens the range of applications in modern signal analysis.

The offset linear canonical transform (OLCT) is a six-parameter linear integral transform that extends the four-parameter linear canonical transform (LCT) by incorporating two additional parameters, a time shift and a frequency modulation\cite{xqq,dwei}. These extra degrees of freedom make the OLCT more general and flexible, enabling it to model a wide variety of electrical and optical signal systems. The OLCT is also known as the special affine Fourier transform and the non-homogeneous linear canonical transform\cite{zhi}. Over the years, the OLCT has attracted significant research interest, particularly in extending well-known time–frequency analysis results from the Fourier transform and LCT domains. Consequently, a comprehensive theoretical encompassing sampling theorems, convolution and correlation properties, eigen functions, energy concentration problems, generalized prolate spheroidal wave functions, and spectral analysis has been established for the OLCT\cite{gm,multi}. However, the OLCT cannot capture local LCT-frequency information because it uses a global kernel. H. Huo introduced the windowed offset linear canonical transform (WOLCT)\cite{Hau}. The WOLCT is obtained by replacing the LCT kernel with the OLCT kernel, which makes it more general and flexible than the WLCT. Furthermore, several fundamental properties of the WOLCT have been derived by Gao et al\cite{Gaul}. Later shah et al. have discussed the orthogonality relation, energy preserving relation, inversion formula, and range theorem for windowed SAFT\cite{shah stolct}. More recently, a short-time special affine Fourier transform has been proposed based on this new convolution structure for the special affine Fourier transform, and its orthogonality and invertibility properties have been studied within a Hilbert space\cite{lakshman}.

Recently, Lone et al.,\cite{A.K verma} introduced a novel short-time linear canonical transform (ST-LCT) based on the convolution structure of the LCT to localize the linear canonical spectrum over finite time intervals using a window function. The transform reduced computational complexity, established key theoretical properties, and supported applications. Motivated by this work and the properties of OLCT, which includes offset (time-shift and frequency-shift) parameters in its kernel, so signals with shifts or modulations can be reconstructed directly without additional compensation steps, in this paper, we discuss the following aspects:
 \begin{itemize}
     \item Introducing a new convolution-based Short-Time Offset Linear Canonical Transform (STOLCT) by applying convolution in the OLCT domain. 
     \item Deriving fundamental properties of the proposed STOLCT, including orthogonality relations, inversion formula, and a range characterization theorem.
     \item Derivation of the convolution theorem for the STOLCT.
     \item Providing numerical illustrations using Gaussian functions and rectangular window functions..
     \item Establishing the Poisson summation formula, Paley–Wiener criterion, and sampling theorem based on the proposed STOLCT
      \item Demonstrating Numerical and graphical error comparison between ST-LCT and STOLCT sampling methods.
 \end{itemize}
 This paper is arranged as follows: Section 2 reviews preliminaries. Section 3 develops the convolution-based STOLCT and its theoretical properties. Section 4 presents potential applications, numerical results, graphical representations, and comparison with ST-LCT sampling. Section 5 concludes the paper.

	\section{Preliminaries}
	This section reviews the basic definitions of the Fourier transform (FT),offset linear canonical transform (OLCT), and the windowed Fourier transform (WFT), OLCT convolution that will be used later.
		\begin{definition}
		Let the function $f \in L^1(\mathbb{R})$. Then Fourier transform of $f$  is  defined as \cite{ap}
		\begin{eqnarray}
			({F}f)(u) = \frac{1}{\sqrt{2\pi}} \int_{\mathbb{R}} e^{-ixu}f(x)dx,\hspace{3mm} \forall u\in \mathbb{R}.
		\end{eqnarray}
		The inverse Fourier transform can be written as
		\begin{eqnarray}
			f(x)=\frac{1}{\sqrt{2\pi}} \int_{\mathbb{R}} e^{ixu}({F}f)(u)du,\hspace{3mm} \forall x\in \mathbb{R}.
		\end{eqnarray}
	\end{definition}
		\begin{definition}\label{OLCT}
        Let $f\in   L^1(\mathbb{R}) $ then OLCT of $f$ can be defined as   \cite{gm} 
		\begin{eqnarray}\label{2.3}
			\small
			O_M[f(t)](u)=
			\begin{cases}
				\int_{\mathbb{R}}f(t)\mathcal{K}_M(t,u)dt & \hspace{-4mm}, ~~ b\neq0\\
				\sqrt{d}~e^{i\frac{cd}{2}(u-u_0)^2+i\omega_0}f[d(u-u_0)] & \hspace{-0.3cm}, ~~b=0, 
			\end{cases}
		\end{eqnarray}
		where
		\begin{eqnarray}\label{definition olct }
			\mathcal{K}_M (t ,u)= \sqrt{\frac{1}{2\pi i b}}~e^{\frac{i}{2b}(a t^2 + 2 t (u_0-u)-2u(du_0-b\omega_0)+ du^2+ du_0^2)},
		\end{eqnarray}
		is the kernel of OLCT, with	$ a, b, c, d, u_0, \omega_0 \in \mathbb{R}$ and $ad - bc = 1$. Here, we will restrict our focus to the OLCT case where 
		$b\neq0$ .
		Then its inverse transformation can be defined as 
		\begin{eqnarray}\label{inverse olct}
			f(x)=  \Big(\mathcal{O}_{M^{-1}} (\mathcal{O}_{M} f)\Big)(x)= C \int_{\mathbb{R}}  \mathcal{K}_{M^{-1}} (x, u) (\mathcal{O}_{M} f)(u) du,
		\end{eqnarray}
		where 
		\begin{eqnarray*}
			&&{M^{-1}=(d,-b,-c,a,b\omega_0-du_0,cu_0-a\omega_0) }
			~\text{and} \\ && C= e^{\frac{i}{b}(cdu_0^2-2adu_0\omega_0+ab\omega_0^2)}.
		\end{eqnarray*}
		
	\end{definition}
	\begin{definition}
		 Let $f\in   L^2(\mathbb{R}) $ with respect to a window function $g\in   L^2(\mathbb{R})\backslash \{0\}$ then Windowed Fourier transform can be defined as\cite{A.K verma}
		  \begin{eqnarray}
		 	\mathcal{F}_w[f](u,\omega)= \frac{1}{\sqrt{2\pi}}  \int_{\mathbb{R}} f(x)\overline{g(x-u)}e^{-i\omega x}dx,
		 \end{eqnarray}
		 where $u, \omega \in \mathbb{R}.$
		 \begin{eqnarray*}
		 		\mathcal{F}_w[f](u,\omega) &=&  \frac{1}{\sqrt{2\pi}}  \int_{\mathbb{R}} f(x)\overline{g(x-u)}e^{-i\omega x}dx\\
		 		&=&\frac{1}{\sqrt{2\pi}}  \int_{\mathbb{R}} \tau_{-\omega}f((x))\overline{g'(u-x)}dx\\
		 		&=& ((\tau_{-\omega}f)\star \bar{g'})(u).
		 \end{eqnarray*}
		 where $g'=g(-x)$ and $\tau_{\omega}$ is the usual modulation operator and $\oplus $ denotes the classical convolution operator.
	\end{definition}
	
    \begin{definition}
For any function $f(x),g(x)\in L^2(\mathbb{R})$,let us define
a convolution operation $\oplus$ of OLCT as follows \cite{xqq}
\begin{eqnarray} 
(f\oplus g )(x)= \frac{1}{\sqrt{2\pi i b}}\int_{\mathbb{R}}f(x) {g(u-x)}e^{\frac{-iax}{b}(u-x)+\frac{i}{2b}du_0^2}dx. 
	\end{eqnarray}
    and when apply OLCT on this convolution we get 
    \begin{eqnarray}\label{olct convolution}
        \mathcal{O}_M[(f\oplus g )(x)]v=\mathcal{O}_M[f](v)\mathcal{O}_M[g](v)e^{\frac{i}{2b}[-dv^2+2v(du_0-b\omega_0)]}
    \end{eqnarray}
    \end{definition}
		\section{STOLCT using a convolution approach}
In this section, we develop a convolution-based formulation of the short-time offset linear canonical transform (STOLCT). We illustrate the proposed method with graphical and numerical examples using Gaussian and rectangular windows. Finally, we present  supporting lemmas, the orthogonality properties, the main theorems, associated with the transform.
	\begin{definition}
 Let $f\in   L^2(\mathbb{R}) $ with respect to a window function $g\in   L^2(\mathbb{R})\backslash \{0\}$ then  OLCT transform can be defined as
	\begin{eqnarray}\label{defined convolution} 
\mathcal{O}_{g_w}[f](u,\eta)=((\mathcal{M}_{-\tau})\otimes g')(u)=  \frac{1}{\sqrt{2\pi i b}}\int_{\mathbb{R}}f(x) \overline{g(x-u)}e^{\frac{-iax}{b}(u-x)+\frac{i}{2b}du_0^2-i\tau x}dx. 
	\end{eqnarray}
	\end{definition}
	This Definition enables us to make following observations:
	\begin{itemize}
		\item \begin{eqnarray}\label{inner product form stolct}
		    \mathcal{O}_{g_w}[f](u,\eta)= \frac{1}{\sqrt{2\pi i b}}<f, g_{u,\eta}^M\rangle = \frac{1}{\sqrt{2 \pi b i  }} \int_{\mathbb{R}} f(t) \overline{ g_{u,\tau}^M(t)}dt, \\
		where \quad  g_{u,\eta}^M(t)=g(t-u)~e^{\frac{iat}{b}(u-t)-\frac{i}{2b}du_0^2+i\tau t}\label{analysing}
		\end{eqnarray} 
	\end{itemize}
		\begin{itemize}
		\item The analyzing function\eqref{analysing} satisfy norm equality:
		$ ||g_{u,\eta}^M(t)||=||g(t-u)||$
	\end{itemize}
	\begin{itemize}
		\item The relation between STOLCT and STFT is 
		\begin{eqnarray}
				\mathcal{O}_{g_w}[f](u,\tau)=e^{\frac{-i}{2b}du_0^2} \mathcal{F}_w[F(t)](u,\tau)\\
				where \quad F(x)= e^{\frac{-iat}{b}(u-t)}\nonumber
		\end{eqnarray}
	\end{itemize}

	\begin{example}\label{example}
		Consider a rectangular window function defined 
		\begin{eqnarray}
		g(x)=\begin{cases}
			1, & \hspace{3mm}  |x|\leq \frac{T}{2} \\
			0, & \hspace{3mm}  otherwise.
		\end{cases}
				\end{eqnarray}
				and $f(x)= e^{-\alpha t^2}$ where $\alpha> 0$
	\end{example}
	we have convolution based STOLCT i.e. 
		\begin{eqnarray} 
		\mathcal{O}_{g_w}[f](u,\eta) =  \frac{1}{\sqrt{2\pi i b}}\int_{\mathbb{R}}f(t) \overline{g(t-u)}e^{\frac{-iat}{b}(u-t)+\frac{i}{2b}du_0^2-i\eta t}dt 
	\end{eqnarray}
	putting $f(t)$ and $g(t)$ in Convolution based STOLCT  we get
	\begin{eqnarray*}
			\mathcal{O}_{g_w}[f](u,\eta) &= & \frac{1}{\sqrt{2\pi i b}}\int_{u-\frac{T}{2}}^{u+\frac{T}{2}}e^{-\alpha t^2}\times 1\times e^{\frac{-iat}{b}(u-t)+\frac{i}{2b}du_0^2-i\eta t}dt\\
			&=&\frac{e^{\frac{i}{2b}du_0^2}}{\sqrt{2\pi i b}}\int_{u-\frac{T}{2}}^{u+\frac{T}{2}}e^{-(\alpha-i\frac{ai}{b})t^2-t(i\eta+\frac{iau}{b})}dt.
	\end{eqnarray*}
\begin{figure}[h]
\centering
\begin{tabular}{cc}
\includegraphics[width=0.48\linewidth]{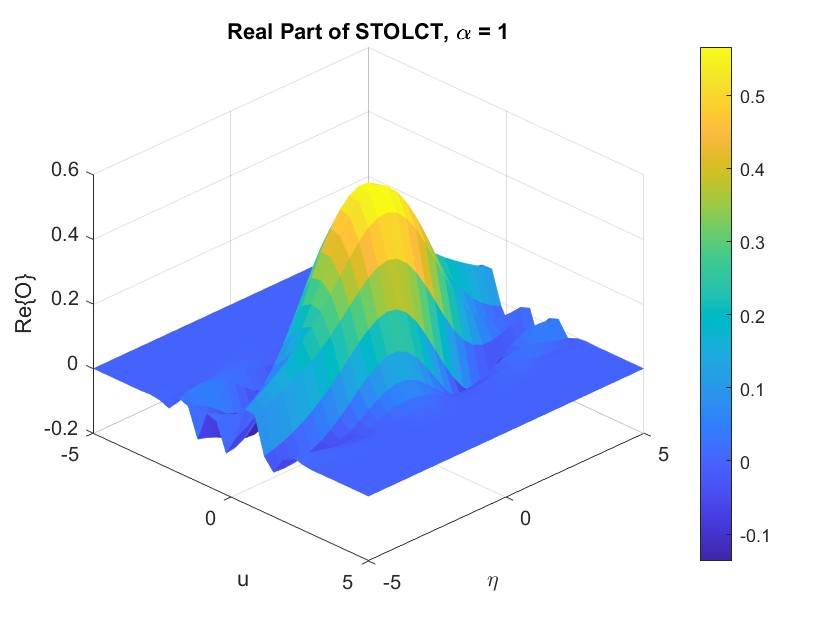} &
\includegraphics[width=0.48\linewidth]{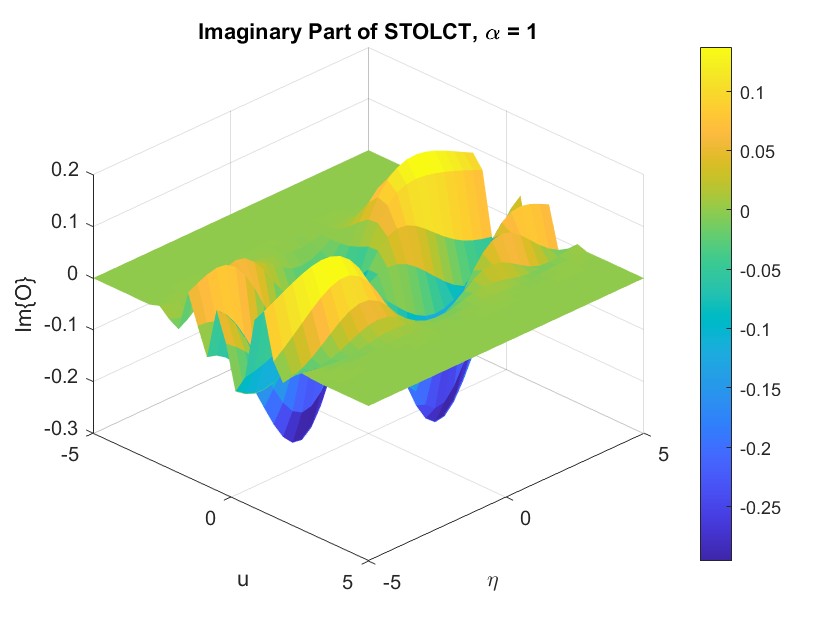} \\
(i) Real part & (ii) Imaginary part \\
\end{tabular}
\caption{Real and Imaginary components for $\alpha = 1$ for example \ref{example}}
\label{fig:alpha1_example1}
\end{figure}

\textbf{Numerical simulation:}
 We present numerical simulations of the proposed transform using a Gaussian signal $f(x)= e^{-\alpha t^2}$
and a rectangular window function \ref{example}. The corresponding graphs illustrate the behavior of the real and imaginary parts for different admissible values of $\alpha>0$
\begin{table}[h]
\centering
\caption{Real and Imaginary values for different $\alpha$ in STOLCT }
\label{tab:stolct_alpha}
\begin{tabular}{|c|c|c|}
\hline
\textbf{$\alpha$} & \textbf{Real Part} & \textbf{Imaginary Part} \\
\hline
1.5  & 0.50332 & 0.23725 \\
\hline
0.85 & 0.58803 & 0.32261 \\
\hline
0.65 & 0.62112 & 0.36784 \\
\hline
0.45 & 0.65854 & 0.42941 \\
\hline
0.25 & 0.70107 & 0.51609 \\
\hline
0.01 & 0.76021 & 0.74021 \\
\hline
\end{tabular}
\end{table}

	\begin{figure}[h]
\centering
\begin{tabular}{cc}
\includegraphics[width=0.48\linewidth]{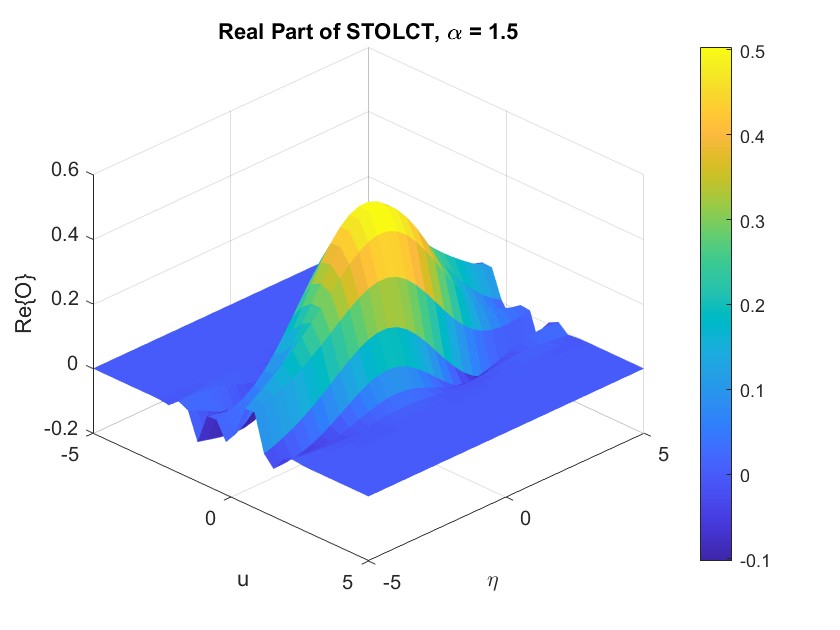} &
\includegraphics[width=0.48\linewidth]{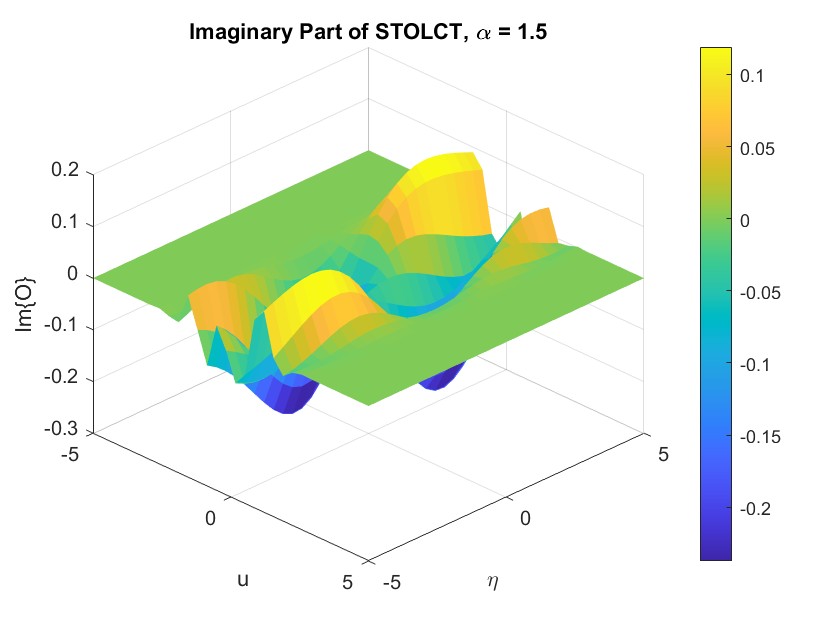} \\
(a)  & (b)  \\
\end{tabular}
\caption{Real and Imaginary components $\alpha = 1.5$}
\label{fig:gauss_rect}
\end{figure}

 \begin{figure}[h]
\centering
\begin{tabular}{cc}
\includegraphics[width=0.48\linewidth]{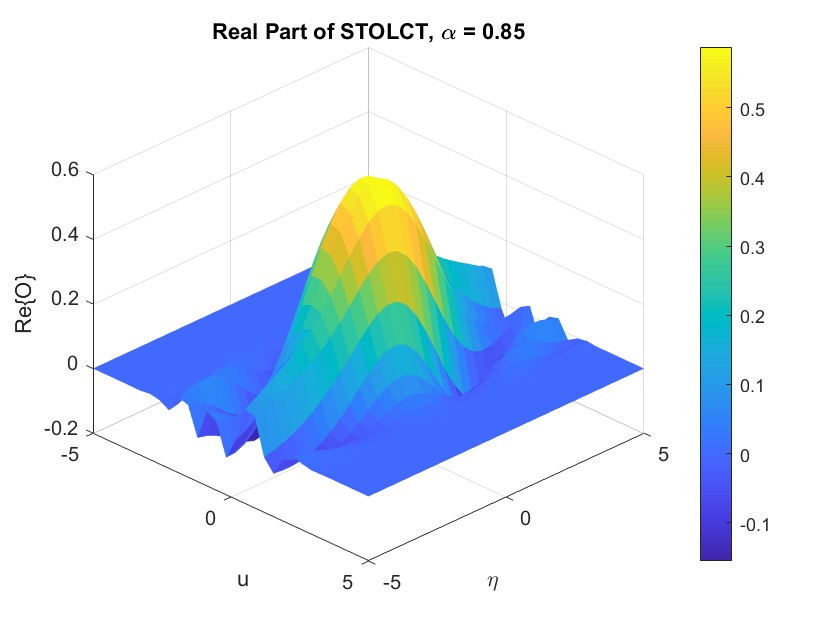} &
\includegraphics[width=0.48\linewidth]{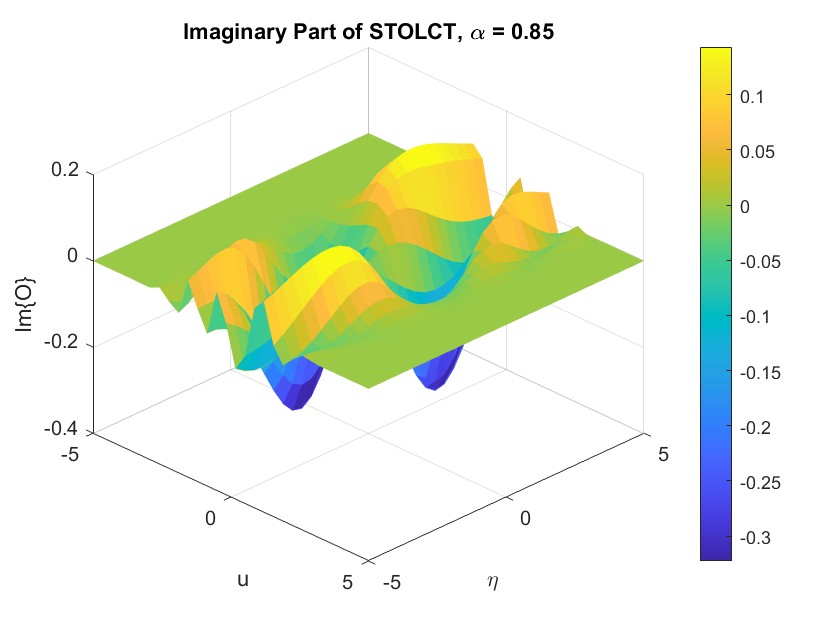} \\
(c) & (d)  \\
\end{tabular}
\caption{Real and Imaginary components $\alpha = 0.85$}
\label{fig:real_imaginary}
\end{figure}

	\begin{figure}[h]
\centering
\begin{tabular}{cc}
\includegraphics[width=0.48\linewidth]{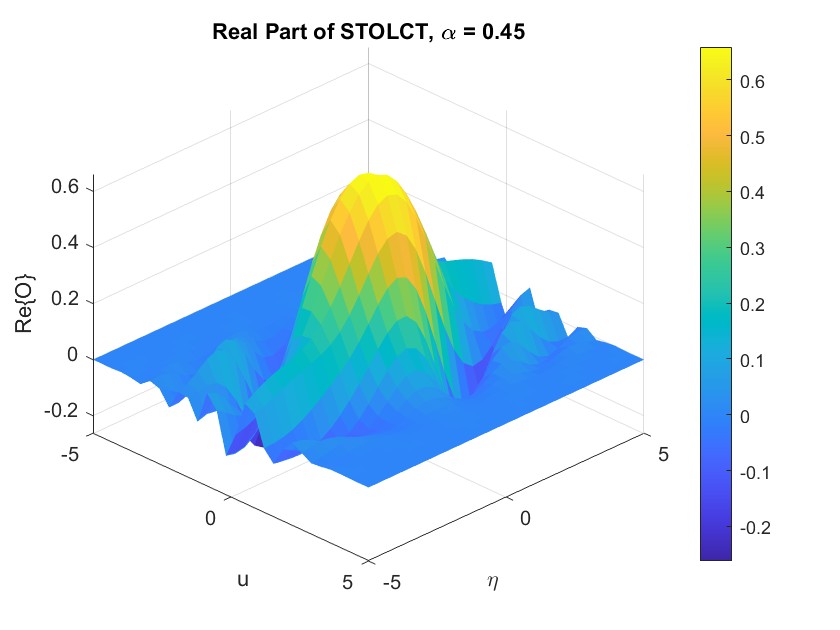} &
\includegraphics[width=0.48\linewidth]{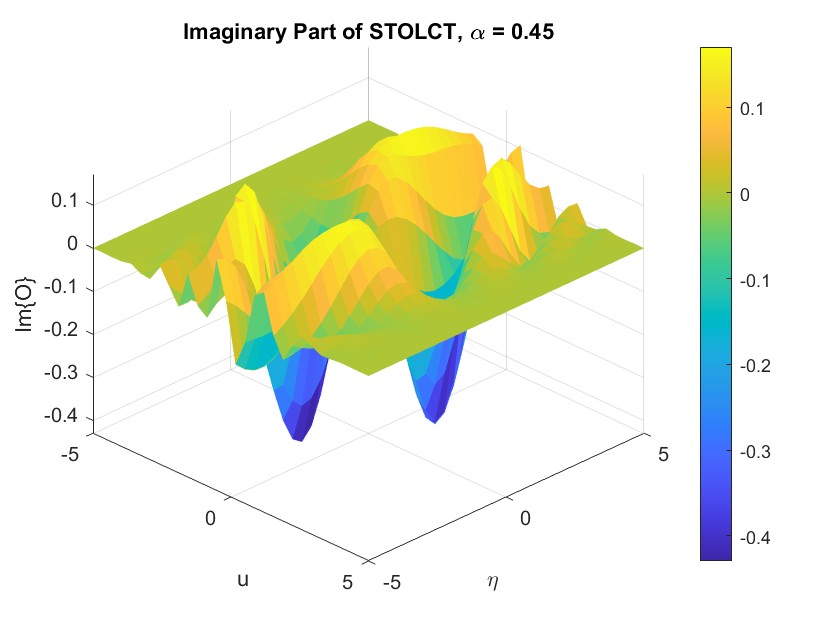} \\
(e)  & (f)  \\
\end{tabular}
\caption{Real and Imaginary components $\alpha = 0.45$}
\label{fig:real_imag_045}
\end{figure}

	\begin{figure}[h]
\centering
\begin{tabular}{cc}
\includegraphics[width=0.48\linewidth]{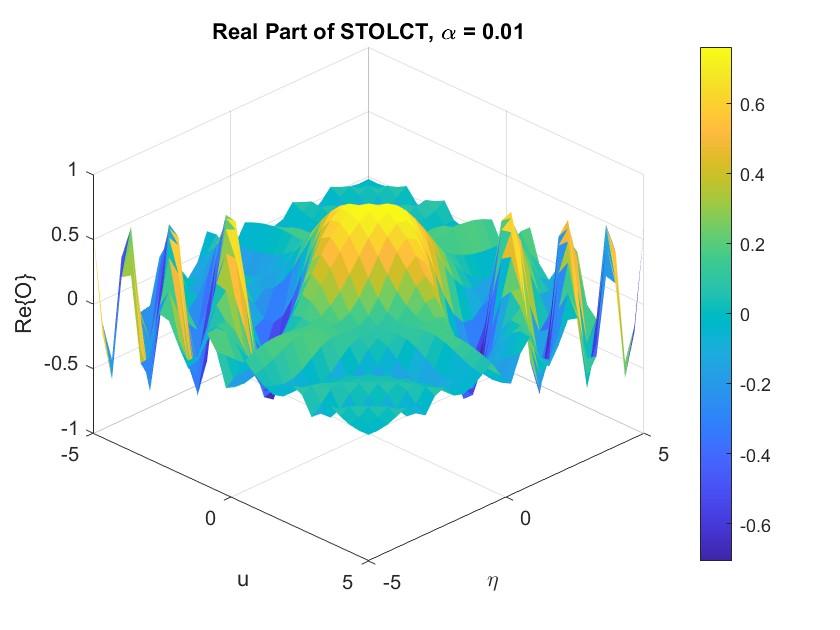} &
\includegraphics[width=0.48\linewidth]{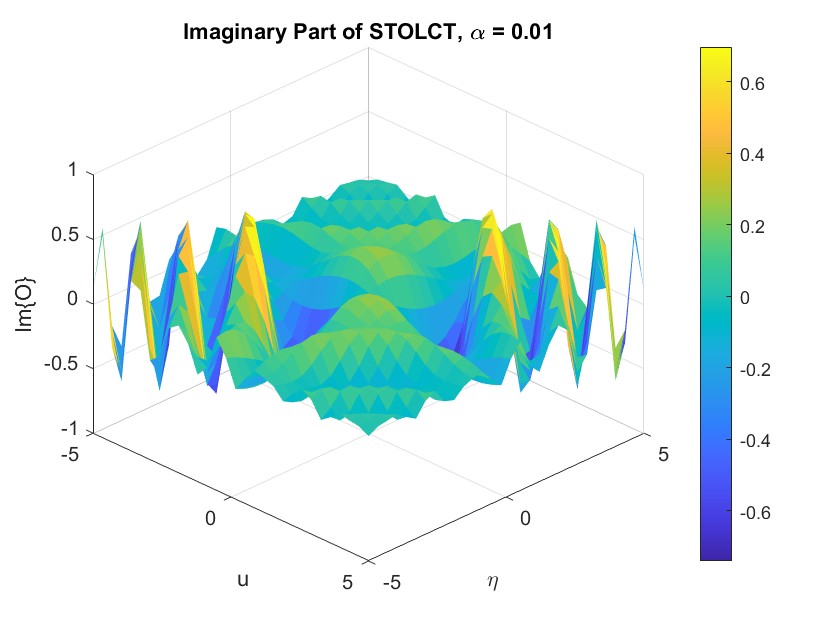} \\
(g)  & (h)  \\
\end{tabular}
\caption{Real and Imaginary components for $\alpha = 0.01$}
\label{fig:alpha001}
\end{figure}

	\begin{remark}
From the Table \ref{tab:stolct_alpha} and figures \ref{fig:gauss_rect}, \ref{fig:real_imaginary},    \ref{fig:real_imag_045}, \ref{fig:alpha001} we observed that as parameter  decreases it is observed that both the real and imaginary parts depend strongly on the parameter 
$\alpha$ decreases, the real part increases gradually while the imaginary part increases more rapidly, and the two values become closer to each other. This indicates that for suitably small 
$\alpha$ the real and imaginary components can be nearly the same, so the transform output may approach a balanced or almost constant complex value.
	\end{remark}
	\FloatBarrier
	\begin{lemma}
 Let  $f\in L^1(\mathbb{R})$ and M be the OLCT parametric matrix. Then the convolution-based STOLCT 
 $\mathcal{O}_{g_w}[f](u,\tau) $ is continuous on $\mathbb{R}^2$.
	\end{lemma}
	\begin{proof}
		\begin{eqnarray*}
&&|\mathcal{O}_{g_w}[f](u,\tau+h)-\mathcal{O}_{g_w}[f](u,\tau)|\\ &=&| \frac{1}{\sqrt{2\pi i b}}\int_{\mathbb{R}}f(x) \overline{{g}(x-u)}e^{\frac{-iax}{b}(u-x)+\frac{i}{2b}du_0^2-i(\tau+h) x}dx- \frac{1}{\sqrt{2\pi i b}}\int_{\mathbb{R}}f(x) \overline{{g}(x-u)}e^{\frac{-iax}{b}(u-x)+\frac{i}{2b}du_0^2-i\tau x}|\\
&=&\frac{1}{\sqrt{2\pi |b|}}|\int_{\mathbb{R}}f(x) \overline{{g}(x-u)}e^{\frac{-iax}{b}(u-x)+\frac{i}{2b}du_0^2}~(e^{-i x(\tau +h)}- e^{-i\tau x })|\\
&=& \frac{1}{\sqrt{2\pi |b|}}|\int_{\mathbb{R}}f(x) \overline{{g}(x-u)}e^{\frac{-iax}{b}(u-x)+\frac{i}{2b}du_0^2-i\tau x}(e^{-i x h}- 1)|.
		\end{eqnarray*}
Taking the limit as $ {h \to 0}$, we get \\
		\begin{eqnarray*}
		\lim_{h\to 0}|\mathcal{O}_{g_w}[f](u,\tau+h)-\mathcal{O}_{g_w}[f](u,\tau)| = 0
		\end{eqnarray*} 
        Therefore, $\mathcal{O}_{g_w}[f](u,\tau)$ is continuous with respect to $\tau$.  
Similarly, continuity with respect to $u$ can be proved in the same manner.  
Hence, the convolution-based STOLCT is continuous on $\mathbb{R}^2$.
	\end{proof}
	\begin{lemma}
		Reimann Lebesgue lemma : If $f \in L^1(\mathbb{R})\cap L^2(\mathbb{R}) $ and the windowed function $g\in L^2(\mathbb{R})$ then, we have 	
		\begin{eqnarray*}
		\lim_{|\tau|\to \infty } |\mathcal{O}_{g_w}[f](u,\tau)| = 0 
		\end{eqnarray*}
	\end{lemma}
	\begin{proof}
		As $e^{-i\tau x } = - e^{i \tau (x+\frac{\pi}{\tau})}$, we have 
		\begin{eqnarray*}
			\mathcal{O}_{g_w}[f](u,\tau) = -\frac{1}{\sqrt{2 \pi i b}}\int_{\mathbb{R}}f(x) \overline{{g}(x-u)}e^{\frac{-iax}{b}(u-x)+\frac{i}{2b}du_0^2-i\tau (x+\frac{\pi}{\tau})}dx.
		\end{eqnarray*}
		Taking $ x+ \frac{\pi}{\tau}= z$, we have
		\begin{eqnarray*}
			\mathcal{O}_{g_w}[f](u,\tau) =	-\frac{1}{\sqrt{2 \pi i b}}\int_{\mathbb{R}}f(z-\frac{\pi}{\tau}) \overline{{g}(z-\frac{\pi}{\tau}-u)}e^{\frac{-ia(z-\frac{\pi}{\tau})}{b}(u-(z-\frac{\pi}{\tau}))+\frac{i}{2b}du_0^2-i\tau  (z)}dz.
		\end{eqnarray*}
		Also, we can write 
		\begin{eqnarray*}
				\mathcal{O}_{g_w}[f](u,\tau)&=&\frac{1}{2}\bigg[	\mathcal{O}_{g_w}[f](u,\tau)+	\mathcal{O}_{g_w}[f](u,\tau)\bigg]\\
				&=&\frac{1}{2}\bigg[\frac{1}{\sqrt{2\pi i b}}\int_{\mathbb{R}}f(x) \overline{{g}(x-u)}e^{\frac{-iax}{b}(u-x)+\frac{i}{2b}du_0^2-i(\tau) x}dx 
				\\&+&\big(-\frac{1}{\sqrt{2 \pi i b}}\int_{\mathbb{R}}f(x-\frac{\pi}{\tau}) \overline{{g}(x-\frac{\pi}{\tau}-u)}e^{\frac{-i(ax-\frac{\pi}{\tau})}{b}(u-(x-\frac{\pi}{\tau}))+\frac{i}{2b}du_0^2-i\tau  (x)}dx\big)\bigg].
		\end{eqnarray*}
		Now, taking modulus on both sides, we get 
		\begin{eqnarray*}
		\bigg|	\mathcal{O}_{g_w}[f](u,\tau)\bigg|
		&=&\bigg|\frac{1}{2}\bigg[\frac{1}{\sqrt{2\pi i b}}\int_{\mathbb{R}}f(x) \overline{{g}(x-u)}e^{\frac{-iax}{b}(u-x)+\frac{i}{2b}du_0^2-i(\tau) x}dx 
		\\&-&\big(\frac{1}{\sqrt{2 \pi i b}}\int_{\mathbb{R}}f(x-\frac{\pi}{\tau}) \overline{{g}(x-\frac{\pi}{\tau}-u)}e^{\frac{-i(ax-\frac{\pi}{\tau})}{b}(u-(x-\frac{\pi}{\tau}))+\frac{i}{2b}du_0^2-i\tau  (x)}dx\big)\bigg]\bigg|\\
		&\leq&\frac{1}{2}\bigg|\frac{1}{\sqrt{2\pi i b}}\int_{\mathbb{R}}e^{\frac{-iax}{b}(u-x)}f(x) \overline{{g}(x-u)}-\frac{1}{\sqrt{2 \pi i b}}\int_{\mathbb{R}}e^{\frac{-i(ax-\frac{\pi}{\tau})}{b}(u-(x-\frac{\pi}{\tau}))}f(x-\frac{\pi}{\tau}) \overline{{g}(x-\frac{\pi}{\tau}-u)}\bigg|.
		\end{eqnarray*}
          Since $f\in L^1(\mathbb{R})$ and $g\in L^2(\mathbb{R})$, the integrand is absolutely integrable. 
Moreover, as $|\tau|\to\infty$, the factor $e^{-i\tau x}$ oscillates  rapidly. 
Due to this behavior, the positive and negative contributions cancel each other, 
and therefore the integral tends to zero.
		\begin{eqnarray*}
			 \lim_{|\tau|\to \infty } |\mathcal{O}_{g_w}[f](u,\tau)| = 0 
		\end{eqnarray*}
       
		Hence, the Riemann--Lebesgue lemma holds.
	\end{proof}
		\begin{proposition}\label{1st proposition}
		Let $ \mathcal{O}_{g_w}[f](u,\eta)$ be the convolution based STOLCT of any $f\in L^2(\mathbb{R})$, then we have 
		\begin{eqnarray}
		  \mathcal{O}_{g_w}[f](u,\tau) &=&e^{-i\tau[u-\frac{db\tau}{2}-(du_0-b\omega_0)]}
		  \int_{\mathbb R}
		  \overline{\mathcal K_M(\xi,u)}
		  e^{\frac{-i}{2b}[-d\xi^2+2db\tau\xi
		  	]}\\&\times&
		  \overline{\mathcal O_M[\tilde g](\xi-b\tau)}
		  \mathcal O_M[f](\xi)
		  \, d\xi\\ where \quad \tilde{g}=e^{\frac{-i}{b}ax^2}
		\end{eqnarray}
	
	\end{proposition}
	\begin{proof}
		Using equations \eqref{defined convolution} and \eqref{olct convolution} we can write,
		 \begin{eqnarray}\label{olct on windowed}
		\mathcal{O}_M[ \mathcal{O}_{g_w}[f](u,\eta)](v)=\mathcal{O}_M[\mathcal{M}_{-\tau}f](v)\mathcal{O}_M[\bar{g'}](v)~e^{\frac{i}{2b}[-dv^2+2v(du_0-b\omega_0)]}\\
        where \quad g'=g(-t).
			\end{eqnarray}
			 From the Definition \eqref{OLCT} we get,
			\begin{eqnarray*}
				\mathcal{O}_M[\bar{g'}](v)=\nonumber \sqrt{\frac{1}{2\pi i b}}\int_{\mathbb{R}}~e^{\frac{i}{2b}(a t^2 + 2 t (u_0-v)-2v(du_0-b\omega_0)+ dv^2+ du_0^2)}\overline{g(-t)}dt.\nonumber
			\end{eqnarray*}
			Putting $-t=x$, we get
				\begin{eqnarray}\label{conjugate (g(x))}
				\mathcal{O}_M[\bar{g'}](v)&=&\nonumber \sqrt{\frac{1}{2\pi i b}}\int_{\mathbb{R}}~e^{\frac{i}{2b}(a x^2 - 2 x (u_0-v)-2v(du_0-b\omega_0)+ dv^2+ du_0^2)}\overline{g(x)}dx.\\\nonumber
				&=&\sqrt{\frac{1}{2\pi i b}}\int_{\mathbb{R}}~~e^{\frac{-i}{2b}(a x^2 + 2 x (u_0-v)-2v(du_0-b\omega_0)+ dv^2+ du_0^2)}\overline{g(x)}\\\nonumber
				&\times& e^{\frac{i}{2b}[(2ax^2)-4v(du_0-b\omega_0)+2dv^2+2du_0^2]}dx \\
				&=& 	\overline{\mathcal{O}_M[{\tilde{g}}](v)}~e^{\frac{i}{2b}[-4v(du_0-b\omega_0)+2dv^2+2du_0^2]},
                where\quad \tilde{g}(x)= e^{\frac{-i}{b}ax^2}.
			\end{eqnarray}
Now,
			\begin{eqnarray}\label{modulation of olct}
				\mathcal{O}_M[\mathcal{M}_{-\tau}f](v)&=& \nonumber \sqrt{\frac{1}{2 \pi i b}} \int_{\mathbb{R}}~e^{\frac{i}{2b}(a t^2 + 2 t (u_0-v)-2v(du_0-b\omega_0)+ dv^2+ du_0^2)}(\mathcal{M}_{-\tau}f) dt\\\nonumber
				&=&\sqrt{\frac{1}{2 \pi i b}} \int_{\mathbb{R}}~e^{\frac{i}{2b}(a t^2 + 2 t (u_0-v)-2v(du_0-b\omega_0)+ dv^2+ du_0^2)}~e^{-i\tau t}f(t) dt\\ \nonumber
				&=&\sqrt{\frac{1}{2 \pi i b}} \int_{\mathbb{R}}~e^{\frac{i}{2b}(a t^2 + 2 t (u_0-(v+b\tau))-2v(du_0-b\omega_0)+ dv^2+ du_0^2)}f(t) dt\\ \nonumber
				&=&\sqrt{\frac{1}{2 \pi i b}} \int_{\mathbb{R}}~e^{\frac{i}{2b}(a t^2 + 2 t (u_0-(v+b\tau))-2(v+b\tau)(du_0-b\omega_0)+ d(v+b\tau)^2+ du_0^2)}f(t) dt\\ \nonumber
				&\times&e^{\frac{i}{2b}[2b\tau(du_0-b\omega_0) -db^2\tau^2-2dvb\tau]}\\
				&=&\mathcal{O}_M[f(t)](v+b\tau)e^{\frac{i}{2b}[2b\tau(du_0-b\omega_0) -db^2\tau^2-2dvb\tau]}.
			\end{eqnarray} 
Substituting equations \eqref{conjugate (g(x))},\eqref{modulation of olct} into \eqref{olct on windowed}, we obtain
			\begin{eqnarray*}
					\mathcal{O}_M[ \mathcal{O}_{g_w}[f](u,\tau)](v)&=&\overline{\mathcal{O}_M[{\tilde{g}}](v)}~e^{\frac{i}{2b}[-4v(du_0-b\omega_0)+2dv^2+2du_0^2]}\\&\times&\mathcal{O}_M[f](v+b\tau) \times e^{\frac{i}{2b}[2b\tau(du_0-b\omega_0) -db^2\tau^2-2dvb\tau]}~e^{\frac{i}{2b}[-dv^2+2v(du_0-b\omega_0)]}\\&=&\overline{\mathcal{O}_M[{\tilde{g}}](v)}\mathcal{O}_M[f](v+b\tau)~e^{\frac{i}{2b}[(-2v+2b\tau)(du_0-b\omega_0)-d(v+b\tau)^2+2dv^2+2du_0^2]}.
			\end{eqnarray*}
			Applying \eqref{inverse olct} we get  
			\begin{eqnarray*}
				\mathcal{O}_{g_w}[f](u,\tau)&=&  \int_{\mathbb{R}}~\overline{\mathcal{K}_M(u,v)}  \overline{\mathcal{O}_M[{\tilde{g}}](v)}~e^{\frac{i}{2b}[(-2v+2b\tau)(du_0-b\omega_0)-d(v+b\tau)^2+2dv^2+2du_0^2]}\mathcal{O}_M[f](v+b\tau)dv\\&=&  \sqrt{\frac{1}{2 \pi (-i) b}} \int_{\mathbb{R}}~e^{\frac{-i}{2b}[au^2+2u(u_0-v)-2v(du_0-b\omega_0)+dv^2+du_0^2]}\\ 
				&\times&e^{\frac{i}{2b}[(-2v+2b\tau)(du_0-b\omega_0)-d(v+b\tau)^2+2dv^2+2du_0^2]}\overline{\mathcal{O}_M[{\tilde{g}}](v)}\mathcal{O}_M[f](v+b\tau)dv\\
				&=&	\sqrt{\frac{1}{2 \pi (-i) b}} \int_{\mathbb{R}}e^{\frac{-i}{2b}[au^2+2u(u_0-v)-du_0^2-dv^2+d(v+b\tau)^2-2b\tau(du_0-b\omega_0)]}\\&\times&\overline{\mathcal{O}_M[{\tilde{g}}](v)}\mathcal{O}_M[f](v+b\tau)dv.
				\end{eqnarray*}
			Replacing $ v+b\tau = \xi$, we get
			\begin{eqnarray*}
				\mathcal{O}_{g_w}[f](u,\tau)&=&\sqrt{\frac{1}{2 \pi (-i) b}}
				\int_{\mathbb{R}}
				e^{\frac{-i}{2b}\Big[au^2+2u\big(u_0-(\xi-b\tau)\big)
					-du_0^2-d(\xi-b\tau)^2+d\xi^2-2b\tau(du_0-b\omega_0)\Big]}\\&\times&\overline{\mathcal{O}_M[\tilde{g}](\xi-b\tau)}
				\mathcal{O}_M[f](\xi)
				\, d\xi\\&=&\int_{\mathbb R}
				\overline{\mathcal K_M(\xi,u)}
				\,
				e^{\frac{-i}{2b}\Big[
					-d\xi^2+2db\tau\xi+2ub\tau
					-db^2\tau^2-2b\tau(du_0-b\omega_0)
					\Big]}\\&\times&
				\overline{\mathcal O_M[\tilde g](\xi-b\tau)}
				\mathcal O_M[f](\xi)
				\, d\xi\\
				&=&e^{\frac{-i}{2b}\Big[2ub\tau	-db^2\tau^2	-2b\tau(du_0-b\omega_0)\Big]}
				\int_{\mathbb R}
				\overline{\mathcal K_M(\xi,u)}
				\,
				e^{\frac{-i}{2b}\Big[-d\xi^2+2db\tau\xi
					\Big]}\\&\times&
				\overline{\mathcal O_M[\tilde g](\xi-b\tau)}
				\mathcal O_M[f](\xi)
				\, d\xi\\
			&=&e^{-i\tau\Big[u-\frac{db\tau}{2}-(du_0-b\omega_0)\Big]}
			\int_{\mathbb R}
			\overline{\mathcal K_M(\xi,u)}
			\,
			e^{\frac{-i}{2b}\Big[-d\xi^2+2db\tau\xi
				\Big]}\\&\times&
			\overline{\mathcal O_M[\tilde g](\xi-b\tau)}
			\mathcal O_M[f](\xi)
			\, d\xi.
			\end{eqnarray*}	
              Which completes the proof.
	\end{proof}	
  
	\begin{theorem}
		Let $\mathcal{O}_{g_w}[f_1]$   $\mathcal{O}_{g_w}[f_2]$ denoted the convolution based STOLCTs of $f_1, f_2 \in L^2(\mathbb{R}) $ with respect to the given pair of window functions $g_1, g_2 \in L^2(\mathbb{R})$ and the uni modular Matrix M. Then we have 
		\begin{eqnarray}
		\int_{\mathbb{R}}\int_{\mathbb{R}}\mathcal{O}_{g_w}[f_1](u,\tau)\overline{\mathcal{O}_{g_w}[f_2](u,\tau)} dud\tau = \langle g_1,g_2\rangle \langle f_1, f_2 \rangle.
		\end{eqnarray}
	\end{theorem}
		\begin{proof}
		From the proposition \ref{1st proposition} , we have 
		\begin{eqnarray}\label{1st in xi}
		 \mathcal{O}_{g_w}[f](u,\tau) &=&e^{-i\tau[u-\frac{db\tau}{2}-(du_0-b\omega_0)]}
		\int_{\mathbb R}
		\overline{\mathcal K_M(\xi,u)}
		e^{\frac{-i}{2b}[-d\xi^2+2db\tau\xi]}
		\overline{\mathcal O_M[\tilde g](\xi-b\tau)}
		\mathcal O_M[f](\xi)
		\, d\xi
		\end{eqnarray}
			\begin{eqnarray}\label{2nd in eta}
			\mathcal{O}_{g_w}[f](u,\tau) &=&e^{-i\tau[u-\frac{db\tau}{2}-(du_0-b\omega_0)]}
			\int_{\mathbb R}
			\overline{\mathcal K_M(\eta,u)}
			e^{\frac{-i}{2b}[-d\eta^2+2db\tau\eta]}
			\overline{\mathcal O_M[\tilde g](\eta-b\tau)}
			\mathcal O_M[f](\eta)
			\, d\eta
		\end{eqnarray}
		 From \eqref{1st in xi} and \eqref{2nd in eta}  we can write 
		
		\begin{eqnarray*}
			\mathcal{O}_{g_w}[f_1](u,\tau)\;\overline{\mathcal{O}_{g_w}[f_2](u,\tau)} 
			&=&\int_{\mathbb R} \int_{\mathbb R}
			\overline{\mathcal K_M(\xi,u)} \, \mathcal K_M(\eta,u)
			\, 
			\mathcal O_M[f_1](\xi) \, \overline{\mathcal O_M[f_2](\eta)}
			\overline{\mathcal O_M[\tilde g_1](\xi-b\tau)} \, \mathcal O_M[\tilde g_2](\eta-b\tau) \\
		 \\	& \times&
			e^{
			\frac{-i}{2b}[-d\xi^2 + 2db\tau \xi ]}  
			e^{\frac{i}{2b}[-d\eta^2 + 2db\tau \eta ]}	\, d\xi \, d\eta.
		\end{eqnarray*}
		Now, integrating with respect to u and $\tau$, we have
		\begin{eqnarray*}
			\int_{\mathbb{R}}\int_{\mathbb{R}}
			\mathcal{O}_{g_w}[f_1](u,\tau)\;\overline{\mathcal{O}_{g_w}[f_2](u,\tau)} \, du \, d\tau	&=&\int_{\mathbb R^2} \int_{\mathbb R^2}
			\overline{\mathcal K_M(\xi,u)} \, \mathcal K_M(\eta,u)
			\, 
			\mathcal O_M[f_1](\xi) \, \overline{\mathcal O_M[f_2](\eta)}
			\overline{\mathcal O_M[\tilde g_1](\xi-b\tau)} \\ & \times&\mathcal O_M[\tilde g_2](\eta-b\tau) 
				e^{\frac{i}{2b}[d(\xi^2 -\eta^2)-2db\tau (\xi -\eta)]}  
				 d\xi \,d\eta\, du\, d\tau\\
			&=&\int_{\mathbb R^3} 
			\mathcal O_M[f_1](\xi) \, \overline{\mathcal O_M[f_2](\xi)} \overline{\mathcal O_M[\tilde g_1](\xi-b\tau)} \, \mathcal O_M[\tilde g_2](\xi-b\tau)
			\, d\xi \, d\tau\\
			&\times&\int_{\mathbb{R}}\overline{\mathcal K_M(\xi,u)} \, \mathcal K_M(\eta,u)	e^{\frac{i}{2b}[d(\xi^2 -\eta^2)-2db\tau (\xi -\eta) ]} du\\&=&
			\int_{\mathbb R^3} 
			\mathcal O_M[f_1](\xi) \, \overline{\mathcal O_M[f_2](\xi)} \overline{\mathcal O_M[\tilde g_1](\xi-b\tau)} \, \mathcal O_M[\tilde g_2](\xi-b\tau)
			\, d\xi \, d\tau\\
			&\times&\frac{1}{2\pi b}\int_{\mathbb{R}}e^{\frac{i}{2b}[-a\xi^2+a\eta^2-2\xi(u_0-u)+2\eta(u_0-u)]}e^{\frac{i}{2b}[d(\xi^2 -\eta^2)-2db\tau (\xi -\eta)]} du\\
			&=&\int_{\mathbb R^3} 
			\mathcal O_M[f_1](\xi) \, \overline{\mathcal O_M[f_2](\xi)} \overline{\mathcal O_M[\tilde g_1](\xi-b\tau)} \, \mathcal O_M[\tilde g_2](\xi-b\tau)\, d\xi \, d\tau\\
			&\times&\frac{1}{2\pi b}e^{\frac{i}{2b}[-a\xi^2+a\eta^2-2\xi(u_0)+2\eta(u_0)]} e^{\frac{i}{2b}[d(\xi^2 -\eta^2)-2db\tau (\xi -\eta)]}\int_{\mathbb{R}}e^{\frac{i}{2b}[2u(\xi-\eta)]} du\\&=&\int_{\mathbb R^3} 
			\mathcal O_M[f_1](\xi) \, \overline{\mathcal O_M[f_2](\xi)} \overline{\mathcal O_M[\tilde g_1](\xi-b\tau)} \, \mathcal O_M[\tilde g_2](\xi-b\tau)
			\, d\xi \, d\tau\\
			&\times&\frac{1}{2\pi b}e^{\frac{i}{2b}[-a\xi^2+a\eta^2-2\xi(u_0)+2\eta(u_0)]}\\&\times& e^{\frac{i}{2b}[d(\xi^2 -\eta^2)-2db\tau (\xi -\eta)]}\int_{\mathbb{R}} 2\pi|b|\delta(\xi-\eta)d\xi
			\\&=&\int_{\mathbb R^3} 
			\mathcal O_M[f_1](\xi) \, \overline{\mathcal O_M[f_2](\xi)} \overline{\mathcal O_M[\tilde g_1](\xi-b\tau)} \, \mathcal O_M[\tilde g_2](\xi-b\tau)
			\, d\xi \, d\tau\\	&=&\int_{\mathbb{R}}\mathcal{O}_M[f_1](\xi)\overline{\mathcal{O}_M[f_2](\xi)}d\xi\int_{\mathbb{R}}\overline{\mathcal{O}_M[{\tilde{g_1}}](u-bu)}\mathcal{O}_M[{\tilde{g_2}}](\xi-b\eta)d\xi.
			    \end{eqnarray*}
     So, we can rewrite the above equation as
            \begin{eqnarray*}
			\int_{\mathbb{R}}\int_{\mathbb{R}}\mathcal{O}_{g_w}[f_1](u,\tau)\overline{\mathcal{O}_{g_w}[f_2](u,\tau)} du d\tau&=& \langle g_2,g_1\rangle  \langle f_1, f_2 \rangle.\label{orthoganility}   
		\end{eqnarray*}
Which proves the theorem.
	\end{proof}
	\begin{remark}
		For $f_1=f_2=f$ and $g_1=g_2=g$ the orthogonality relation becomes \begin{eqnarray}
		    \int_{\mathbb{R}^2} |\mathcal{O}_{g_w}[f](u,\tau)|^2du\,d\tau= ||f||^2||g||^2,
		\end{eqnarray}
indicating  that ST-OLCT is an energy-preserving transform.The total energy of the ST-OLCT representation equals the product of the energy of the signal and the window, confirming the stability and consistency of the transform.
	\end{remark}
    
	\begin{theorem} Inversion formula:
    Let $f_1 \in L^{2}(\mathbb{R})$ and let $\mathcal{O}_{g_w}[f_1](u,\tau)$ be STOLCT  with respect to the window function $g_1$.  
If $g_1, g_2 \in L^{2}(\mathbb{R})$ with $\langle g_2, g_1 \rangle \neq 0$, then $f$ can be reconstructed from $\mathcal{O}_{g_w}[f_1]$ by
		\begin{eqnarray*}
			f(t)= \frac{1}{\overline{\sqrt{2 \pi b i}}\langle g_2, g_1\rangle} \int_{\mathbb{R}^2}\mathcal{O}_{g_w}[f](u,\tau) g_{2_{u,\tau}}^M(t)\,du\,d\tau.
		\end{eqnarray*}
	\end{theorem}
	\begin{proof}
		By applying the orthogonality relation, we have 
		\begin{eqnarray*}
			 \langle f_1, f_2 \rangle &=& \frac{1}{ \langle g_2,g_1\rangle } \int_{\mathbb{R}}\int_{\mathbb{R}}\mathcal{O}_{g_w}[f_1](u,\tau)\overline{\mathcal{O}_{g_w}[f_2](u,\tau)} \,du\,d\tau\\
			 &=&\frac{1}{\langle g_2,g_1\rangle} \int_{\mathbb{R}^2}\mathcal{O}_{g_w}[f_1](u,\tau)\overline{\mathcal{O}_{g_w}[f_2](u,\tau)} \,du\,d\tau.
             \end{eqnarray*}
          Using \eqref{inner product form stolct}, we can rewrite   
              \begin{eqnarray*}
		\langle f_1, f_2 \rangle&=&\frac{1}{ \langle g_2,g_1\rangle} \int_{\mathbb{R}^2}\mathcal{O}_{g_w}[f_1](u,\tau)~\overline{ \frac{1}{\sqrt{2 \pi  i b }} \int_{\mathbb{R}} f_2(t) \overline{  {g_2}_{u,\tau}^M(t)}dt}\,du\,d\tau\\
			 &=&\frac{1}{ \langle g_2,g_1\rangle}\overline{\frac{1}{\sqrt{2 \pi  i b }}}\int_{\mathbb{R}^2}\mathcal{O}_{g_w}[f_1](u,\tau)  \int_{\mathbb{R}} \overline{f_2(t) } {g_2}_{u,\tau}^M(t)\,dt\,du\,d\tau\\
			 &=&\frac{1}{ \langle g_2,g_1\rangle}\overline{\frac{1}{\sqrt{2 \pi  i b }}}\int_{\mathbb{R}}\big(\int_{\mathbb{R}^2}\mathcal{O}_{g_w}[f_1](u,\tau) {g_2}_{u,\tau}^M(t) \,du\,d\tau\big)   ~\overline{f_2(t) } dt
			 	\end{eqnarray*}
 Since the above equality holds for all $f_2\in L^2(\mathbb R)$,
by the uniqueness of the inner product we obtain
			 	\begin{eqnarray*}
			 		f_1(t)=\frac{1}{\overline{\sqrt{2 \pi b i}}\langle g_2, g_1\rangle} \int_{\mathbb{R}^2}\mathcal{O}_{g_w}[f_1](u,\tau) g_{2_{u,\tau}}^M(t)\,du\,d\tau.
			 	\end{eqnarray*}               
This completes the proof.		 
	\end{proof}
	\begin{theorem}
		Range Theorem: A function $H\in L^2(\mathbb{R}^2)$ is the convolution based STOLCT of square integrable function  $f\in L^2(\mathbb{R}) $ iff the  following reproducing formula is satisfied .
		\begin{eqnarray}\label{range statement}
		H(u',\tau')= \int_{\mathbb{R}^2}H(u,\tau)\mathcal{K}^M_{g_w}(u,\tau,u',\tau') du\, d\tau,
		\end{eqnarray}
		where
		\begin{eqnarray}
		\mathcal{K}^M_{g_w}(u,\tau,u',\tau')=\frac{1}{2 \pi| b |\,||g||^2}\big \langle g_{u,\tau}^M, g_{u',\tau'}^M \big\rangle
		\end{eqnarray} 
	\end{theorem}
\begin{proof}
	 Assume that a function $H\in L^2(\mathbb{R}^2)$ is the STOLCT of $f_1\in L^2(\mathbb{R})$ with respect to the window function $g\in L^2(\mathbb{R})$ and the parametric matrix M ,
	 $\mathcal{O}_{g_w}[f_1]=H$
	 we have \begin{eqnarray*}
	 	H(u',\tau')&=&\mathcal{O}_{g_w}[f_1](u',\tau')\\
	 	&=&\frac{1}{\sqrt{2\pi i b}}\int_{\mathbb{R}}f_1(t)\overline{g_{u',\tau'}^M(t)}dt
	 \end{eqnarray*}
	Using the inversion formula, we obtain 
	 \begin{eqnarray*}
	 	H(u',\tau')&=&\frac{1}{{{2 \pi |b |}}\langle g_2, g_1\rangle} \int_{\mathbb{R}}\int_{\mathbb{R}^2}\mathcal{O}_{g_w}[f_1](u,\tau) g_{2_{u,\tau}}^M(t)\,du\,d\tau~\overline{g_{u',\tau'}^M(t)}\,dt\\
	 	&=&\frac{1}{{2 \pi |b |}\langle g_2, g_1\rangle} \int_{\mathbb{R}^2}\mathcal{O}_{g_w}[f_1](u,\tau)\,du\,d\tau~\int_{\mathbb{R}} g_{u,\tau}^M(t)\overline{g_{u',\tau'}^M(t)}dt\\
	 	&=&\frac{1}{{2 \pi |b|}\langle g, g_1\rangle}\int_{\mathbb{R}^2}\mathcal{O}_{g_w}[f_1](u,\tau)dud\tau\big \langle g_{u,\tau}^M, g_{u',\tau'}^M \big\rangle
	 	 \end{eqnarray*}
         Conversely, assume that a function $H\in L^2(\mathbb{R}^2)$
satisfies \eqref{range statement}. We show that there exists
a function $p \in L^2(\mathbb{R})$ such that
$H=\mathcal{O}_{g_w}[p]$.
	 	 \begin{eqnarray}\label{p(x)}
	 	 	p(t)=\frac{1}{{2 \pi| b |||g||^2}}\int_{\mathbb{R}^2}H(u,\tau) {g}^M_{u,\tau}(t)dud\tau. 
	 	 \end{eqnarray}
	 	 To show that the function $p\in L^2( \mathbb{R})$, we proceed as follows 
	 	 \begin{eqnarray*}
	 	 	||p||^2&=&\int_{\mathbb{R}}p(t)\overline{p(t)}dt\\
	 	 	&=& \int_{\mathbb{R}}\frac{1}{{2 \pi |b |||g||^2}}\int_{\mathbb{R}^2}H(u,\tau) {g}^M_{u,\tau}(t)dud\tau\overline{\frac{1}{{2 \pi |b |\,||g||^2}}\int_{\mathbb{R}^2}H(u',\tau') {g}^M_{u',\tau'}(t)du'd\tau'} dt\\
	 	 	&=&\int_{\mathbb{R}}\frac{1}{(2 \pi| b|)^2\, ||g||^4}\int_{\mathbb{R}^4}H(u,\tau) \overline{H(u',\tau')}\,du\,d\tau du'd\tau' \int_{\mathbb{R}}{g}^M_{u,\tau}(t) \overline{{g}^M_{u',\tau'}(t)} dt\\
	 	 	&=&\frac{1}{2\pi  |b| \,||g||^2}\int_{\mathbb{R}^4}H(u,\tau) \overline{H(u',\tau')} \mathcal{K}_{g_w}^M(u,\tau,u',\tau')\,du\,d\tau du'\,d\tau'\\
	 	 	&=& \frac{1}{{2\pi  |b|} ||g||^2}\int_{\mathbb{R}^2}\overline{H(u',\tau')}\big(\int_{\mathbb{R}^2}H(u,\tau)\mathcal{K}_{g_w}^M(u,\tau,u',\tau')dud\tau\big) du'd\tau'\\ 
	 	 	&=&\frac{1}{2\pi |b | ||g||^2}\int_{\mathbb{R}^2}\overline{H(u',\tau')}H(u,\tau) du'd\tau'\\
	 	 	&=&\frac{1}{{2\pi | b|} ||g||^2}||H||^2<\infty\\
             Hence\quad p\in L^2(\mathbb{R}).
	 	 	 \end{eqnarray*}
             Furthermore, the STOLCT of the function $p$ can be written as
\begin{eqnarray*}
	 	 	 	\mathcal{O}_{g_w}[p](u',\tau')&=&\int_{\mathbb{R}}p(t)\overline{g_{u',\tau'}^M(t)}dt\\
	 	 	 	&=&\int_{\mathbb{R}}\frac{1}{2 \pi |b|\,||g||^2}\int_{\mathbb{R}^2}H(u,\tau) {g}^M_{u,\tau}(t)dud\tau ~\overline{g_{u',\tau'}^M(t)}dt\\
	 	 	 	&=&\int_{\mathbb{R}^2}H(u,\tau)\frac{1}{2 \pi |b|\,||g||^2}\int_{\mathbb{R}}{g}^M_{u,\tau}(t)\overline{g_{u',\tau'}^M(t)}dt dud\tau\\
	 	 	 	&=&\int_{\mathbb{R}^2}H(u,\tau)\mathcal{K}_{g_w}^M(u,\tau,u',\tau')dud\tau\\
	 	 	 	&=&H(u',\tau').
	 	 	 \end{eqnarray*}
             Thus, the proof is complete.
\end{proof}

\begin{theorem}
	Let $f_1,f_2,g \in L^1(\mathbb{R}) \cap L^2(\mathbb{R})$ and let $M=(a,b,c,d,u_0,\omega_0)$ with $b\neq0$.
	Assume that the window function $g$ satisfies the  $\mathcal{O}_M[\tilde g](v)\neq0  $,  $\mathcal{O}_{g_w}[f_1\otimes f_2](u,\tau)
	=\mathcal{O}_{g_w}[f_1](u,\tau)\,
	\mathcal{O}_{g_w}[f_2](u,\tau).$Then the convolution $f_1\otimes f_2$ admits the inversion formula
	\begin{eqnarray*}
	\begin{aligned}
		(f_1\otimes f_2)(t)
		&= \int_{\mathbb{R}^2}
		\frac{1}{\overline{\mathcal{O}_M[\tilde g](v)}}
		e^{\frac{i}{2b}\left[2v(du_0-b\omega_0)-dv^2-2du_0^2\right]}~e^{it\tau}  \\
		&\quad \times
		\Bigg[
		\frac{1}{\sqrt{2\pi ib}}
		\int_{\mathbb{R}} f_1(x)\overline{g(x-u)}
		e^{\frac{-iax}{b}(u-x)+\frac{i}{2b}du_0^2-i\eta x}\,dx \\
		&\qquad \times
		\frac{1}{\sqrt{2\pi ib}}
		\int_{\mathbb{R}} f_2(y)\overline{g(y-u)}
		e^{\frac{-iay}{b}(u-y)+\frac{i}{2b}du_0^2-i\eta y}\,dy
		\Bigg]
		\,dv\,d\tau .
	\end{aligned}
		\end{eqnarray*}
\end{theorem}
\begin{proof}
	Let $f_1,f_2,g \in L^1(\mathbb{R})\cap L^2(\mathbb{R})$ with $b\neq 0$.  
	By the convolution property of the ST-OLCT, we have
	\[
	\mathcal{O}_{g_w}[f_1\otimes f_2](u,\tau)
	=
	\big(\mathcal{O}_{g_w}[f_1](u,\tau)\,\mathcal{O}_{g_w}[f_2](u,\tau)\big).
	\]
Applying OLCT on both sides, we get
	\begin{eqnarray}
		\mathcal{O}_M[\mathcal{O}_{g_w}[f_1\otimes f_2](u,\tau)](v) &=&	\mathcal{O}_M[\mathcal{O}_{g_w}[f_1](u,\tau)\mathcal{O}_{g_w}[ f_2](u,\tau)] \nonumber\\&=& \mathcal{O}_M[((\mathcal{M}_{-\tau}f_1)\otimes g')(u)((\mathcal{M}_{-\tau}f_2)\otimes g')(u)].
			\end{eqnarray}
			Using Lemma 3.6 we can write
		\begin{eqnarray}
		\mathcal{O}_M[\mathcal{O}_{g_w}[f_1\otimes f_2](u,\tau)](v)&=&\overline{\mathcal{O}_M[{\tilde{g}}](v)}~e^{\frac{i}{2b}[(-2v+2b\tau)(du_0-b\omega_0)-d(v+b\tau)^2+2dv^2+2du_0^2]}\nonumber \\&\times&\mathcal{O}_M[f\otimes f_2](v+b\tau).
		\end{eqnarray}
Comparing the equation 3.15 and 3.14, we get
\begin{eqnarray}\label{compare olct}
 \mathcal{O}_M[((\mathcal{M}_{-\tau}f_1)\otimes g')(u)((\mathcal{M}_{-\tau}f_2)\otimes g')(u)]
&=&\overline{\mathcal{O}_M[{\tilde{g}}](v)}~e^{\frac{i}{2b}[(-2v+2b\tau)(du_0-b\omega_0)-d(v+b\tau)^2+2dv^2+2du_0^2]} \nonumber \\&\times&\mathcal{O}_M[f\otimes f_2](v+b\tau).
\end{eqnarray}
After arranging the \eqref{compare olct}, we get 
\begin{eqnarray*}
\mathcal{O}_M[f_1\otimes f_2](v+b\tau)&=& \mathcal{O}_M[((\mathcal{M}_{-\tau}f_1)\otimes g')(u)((\mathcal{M}_{-\tau}f_2)\otimes g')(u)]\\&\times&\frac{1}{\overline{\mathcal{O}_M[{\tilde{g}}](v)}}~e^{\frac{-i}{2b}[(-2v+2b\tau)(du_0-b\omega_0)-d(v+b\tau)^2+2dv^2+2du_0^2]}.
\end{eqnarray*}
Now, applying inversion of OLCT, and using Fubini’s theorem to interchange the order of integration, we obtain
\begin{eqnarray*}
	(f_1\otimes f_2)(t)
&=&\int_{\mathbb{R}}\mathcal{O}_M[((\mathcal{M}_{-\tau}f_1)\otimes g')(u)((\mathcal{M}_{-\tau}f_2)\otimes g')(u)]\frac{1}{\overline{\mathcal{O}_M[{\tilde{g}}](v)}}\\&\times&e^{\frac{-i}{2b}[(-2v+2b\tau)(du_0-b\omega_0)-d(v+b\tau)^2+2dv^2+2du_0^2]} \mathcal{K}_{M^{-1}} (t,v+b\tau)dv.
\end{eqnarray*}
 By arranging the equation, we get
\begin{eqnarray*}
&&	(f_1\otimes f_2)(t)\\
&=&	\int_{\mathbb{R}} \frac{1}{\overline{\mathcal{O}_M[{\tilde{g}}](v)}}e^{\frac{-i}{2b}[(-2v+2b\tau)(du_0-b\omega_0)-d(v+b\tau)^2+2dv^2+2du_0^2]} \\ &\times&\int_{\mathbb{R}}((\mathcal{M}_{-\tau}f_1)\otimes g')(u)((\mathcal{M}_{-\tau}f_2)\otimes g')(u) \mathcal{K}_M(t,v)\mathcal{K}_{M^{-1}} (t,v+b\tau)dvd\tau\\
&=& \int_{\mathbb{R}^2} \frac{1}{\overline{\mathcal{O}_M[{\tilde{g}}](v)}}e^{\frac{-i}{2b}[(-2v+2b\tau)(du_0-b\omega_0)-d(v+b\tau)^2+2dv^2+2du_0^2]} \\ &\times&\mathcal{K}_M(t,v)\mathcal{K}_{M^{-1}} (t,v+b\tau) \big[\frac{1}{\sqrt{2\pi i b}}\int_{\mathbb{R}}f_1(x) \overline{g(x-u)}e^{\frac{-iax}{b}(u-x)+\frac{i}{2b}du_0^2-i\eta x}dx.\\
&\times& \frac{1}{\sqrt{2\pi i b}}\int_{\mathbb{R}}f_2(y) \overline{g(y-u)}e^{\frac{-iay}{b}(u-y)+\frac{i}{2b}du_0^2-i\eta y}dy\big]dvd\tau\\
&=& \int_{\mathbb{R}^2}\frac{1}{\overline{\mathcal{O}_M[{\tilde{g}}](v)}}e^{\frac{i}{2b}[(2v)(du_0-b\omega_0)-dv^2-2du_0^2]}~e^{it\tau} \big[\frac{1}{\sqrt{2\pi i b}}\int_{\mathbb{R}}f_1(x) \overline{g(x-u)}e^{\frac{-iax}{b}(u-x)+\frac{i}{2b}du_0^2-i\eta x}dx\\
&\times& \frac{1}{\sqrt{2\pi i b}}\int_{\mathbb{R}}f_2(y) \overline{g(y-u)}e^{\frac{-iay}{b}(u-y)+\frac{i}{2b}du_0^2-i\eta y}dy\big]dvd\tau.
\end{eqnarray*}
This completes the proof.
\end{proof} 

\section{Potential Application}
In this section, we discuss several important applications of the convolution-based STOLCT, including the Poisson summation formula, the Paley–Wiener criterion, and the sampling theorem. Numerical simulations and graphical illustrations are also provided to compare the results with existing ones.

\subsection{Poisson summation formula associated with STOLCT}
The Poisson summation formula states that the infinite sum of time-domain samples of a signal $x(t)$ is equivalent to the infinite sum of its spectrum $X(u)$ in the Fourier domain. Mathematically, the Poisson summation formula can be expressed as follows\cite{poisson}
\begin{equation}\label{fourier poisson}
\sum_{k=-\infty}^{\infty} g(t + nT)
=
\frac{1}{T}
\sum_{n=-\infty}^{\infty}
\mathcal{F}[g]\!\left(\frac{n}{T}\right)
e^{j \frac{n}{T} t}
\end{equation}

where $\mathcal{F}[g](u)$ is the traditional Fourier transform of signal
$g(t)$.
\begin{theorem}
	If $\mathcal{O}_{g_w}[f](u,\tau)$ be the convolution based STOLCT then the Poisson summation formula follows for $\mathcal{O}_{g_w}[f](u,\tau)$
	\begin{eqnarray}
		\sum_{n\in \mathbb{Z}}	f(nT) \overline{g(nT-u)}e^{\frac{ia}{2b}(u-(nT))^2}=\frac{1}{CT}e^{\frac{ia}{2b}u^2}\sum_{n\in \mathbb{Z}}\mathcal{O}_{g_w}[f](u,\frac{n}{T}).
	\end{eqnarray}
\end{theorem} 
\begin{proof}
From Definition \eqref{defined convolution} we can write
	\begin{eqnarray*}
	\mathcal{O}_{g_w}[f](u,\tau)&=&  \frac{1}{\sqrt{2\pi i b}}\int_{\mathbb{R}}f(x) \overline{g(x-u)}e^{\frac{-iax}{b}(u-x)+\frac{i}{2b}du_0^2-i\tau x}dx.\\
	&=& \frac{1}{\sqrt{2\pi i b}}\int_{\mathbb{R}}f(x) \overline{g(x-u)}e^{\frac{ia}{2b}(u-x)^2+\frac{ia}{2b}(-u^2)+\frac{i}{2b}du_0^2-i\tau x}dx.\\
\end{eqnarray*}
Rearranging the equation, we get  
\begin{eqnarray}
	e^{\frac{ia}{2b}u^2}\mathcal{O}_{g_w}[f](u,\tau)&=&C\mathcal{F}(f(x) \overline{g(x-u)}e^{\frac{ia}{2b}(u-x)^2})
\end{eqnarray}	
where $C=\frac{e^{\frac{i}{2b}du_0^2}}{\sqrt{ i b}} $
\begin{eqnarray}\label{windowed possian}
	e^{\frac{ia}{2b}u^2}\mathcal{O}_{g_w}[f](u,\frac{n}{T})&=&C\mathcal{F}[f(x) \overline{g(x-u)}e^{\frac{ia }{2b}(u-x)^2}]	(\frac{n}{T})
\end{eqnarray}
As we have from Poisson summation formula \eqref{fourier poisson}  
\begin{eqnarray}\label{possian}
	\sum_{n\in \mathbb{Z}}g(x+nT)=\frac{1}{T}\sum_{n\in \mathbb{Z}}\mathcal{F}[g](\frac{n}{T})e^{\frac{in}{T}x}
\end{eqnarray}
From equation \eqref{possian} and \eqref{windowed possian} we can write
\begin{eqnarray*}
\sum_{n\in \mathbb{Z}}	f(x+nT) \overline{g(x+nT-u)}e^{\frac{ia}{2b}(u-(x+nT))^2}=\frac{1}{CT}e^{\frac{ia}{2b}u^2}\sum_{n\in \mathbb{Z}}\mathcal{O}_{g_w}[f](u,\frac{n}{T})e^{\frac{ik}{T}x}.
\end{eqnarray*}
Which is the desired Poisson summation formula. Moreover, for $x=0$, the summation formula reduces to 
\begin{eqnarray*}
\sum_{n\in \mathbb{Z}}	f(nT) \overline{g(nT-u)}e^{\frac{ia}{2b}(u-(nT))^2}=\frac{1}{CT}e^{\frac{ia}{2b}u^2}\sum_{n\in \mathbb{Z}}\mathcal{O}_{g_w}[f](u,\frac{n}{T}).
\end{eqnarray*}
\end{proof}
\subsection{Paley-wiener theorem associated with STOLCT}
The classical Paley–Wiener theorem, introduced by Raymond Paley and Norbert Wiener, characterizes the Fourier transform of 
$L^2(\mathbb{R})$ functions supported on a symmetric interval as entire functions of exponential type whose restrictions to the real line remain in $L^2(\mathbb{R})$. This result has become a fundamental tool in many transform-based frameworks\cite{paley}.
	
    The physical behavior of any linear time invariant system dictated by the well known Paley Weiner criterion\cite{kou paley} 
	
	\begin{eqnarray}\label{paley}
		\int_{\mathbb{R}}\frac{|\log||\mathcal{F}[f](\tau)|}{1+\tau^2}d\tau< \infty
	\end{eqnarray} 
	The relationship between STOLCT and Classical Fourier transform is 
	\begin{eqnarray}
	e^{\frac{ia}{2b}u^2}\mathcal{O}_{g_w}[f](u,\tau)&=&C\mathcal{F}(f(x) \overline{g(x-u)}e^{\frac{ia}{2b}(u-x)^2})  \nonumber \\
\mathcal{O}_{g_w}[f](u,\tau)	&=&e^{\frac{-ia}{2b}u^2}C\mathcal{F}(f(x) \overline{g(x-u)}e^{\frac{ia}{2b}(u-x)^2})
	\end{eqnarray}	
	Taking Modulus we get
	\begin{eqnarray}\label{relation stolct and ft}
	|\mathcal{O}_{g_w}[f](u,\tau)|	&=&|e^{\frac{-ia}{2b}u^2}C\mathcal{F}[f(x) \overline{g(x-u)}e^{\frac{ia}{2b}(u-x)^2}](\tau)|\\
	&=&|C\mathcal{F}[f(x) \overline{g(x-u)}e^{\frac{ia}{2b}(u-x)^2}](\tau)|.
	\end{eqnarray}
	As the Fourier signal satisfies
	\begin{eqnarray*}
		|\mathcal{F}[f](\tau)|=\frac{1}{\sqrt{2 \pi}}\int_{\mathbb{R}}|f(x)|dx.
	\end{eqnarray*}
	now From the definition of STOLCT we have 
		\begin{eqnarray*}
		|\mathcal{O}_{g_w}[f](u,\tau)|	
		&=&| \frac{1}{\sqrt{2\pi i b}}\int_{\mathbb{R}}f(x) \overline{g(x-u)}e^{\frac{ia}{2b}(u-x)^2+\frac{ia}{2b}(-u^2)+\frac{i}{2b}du_0^2-i\tau x}dx|.\\
		&\leq&\frac{1}{\sqrt{2\pi i b}}||\int_{\mathbb{R}}|f(x) \overline{g(x-u)}|dx.
			\end{eqnarray*}
			from equation \eqref{paley} and \eqref{relation stolct and ft}, we can write 
			\begin{eqnarray*}
			\int_{\mathbb{R}}\frac{|\log||C\mathcal{F}[f(x) \overline{g(x-u)}e^{\frac{ia}{2b}(u-x)^2}](\tau)|}{1+\tau^2}d\tau< \infty
			\end{eqnarray*}
			equivalently, we can write
			\begin{eqnarray}
				\int_{\mathbb{R}}\frac{|\log||\mathcal{O}_{g_w}[f](u,\tau)|}{1+\tau^2}d\tau< \infty.
			\end{eqnarray}

            \FloatBarrier
            
			\subsection{Sampling in the STOLCT domain}Sampling plays a crucial role in signal analysis, as it bridges continuous-time physical phenomena and their discrete-time representations. A classical result in this context is the Shannon sampling theorem, which applies to signals bandlimited in the Fourier domain.

A function $f \in L^{2}(\mathbb{R})$ is said to be $\eta$-band-limited with respect to a linear transform $\mathcal{T}$ if
\begin{equation}
    \mathcal{T}[f](\omega)=0 \quad \text{for all } |\omega|>\eta.
\end{equation}
That is, the transform of $f$ is supported only on the interval $[-\eta,\eta]$.
			
			Mathematically, the classical sampling theorem asserts that band limited signal can be fully reconstructed from its value at uniformly spaced time interval 
			\begin{eqnarray}\label{sampling}
				f(x)=\sum_{n\in\mathbb{Z}}f\big(\frac{n\pi}{\eta}\big)\operatorname{sinc}\big(\frac{\eta}{\pi}( x- n\pi)\big)
				\end{eqnarray}
				A sampling theorem for convolution-based STOLCT is particular interest because many chirp-like signals are not band-limited in the conventional Fourier domains. In this subsection, we will drive shannon's sampling theorem in the context of STOLCT . 
				
				\begin{theorem}\label{sampling theorem}
					If $f\in L^2(\mathbb{R})$ is an $\eta$ band limited signal in the STOLCT domain	Assume that the window function $g$ is such that, for each $u\in\mathbb R$,
					the signal
					\[
					F(x)=f(x)\overline{g(x-u)}
					\]
					is $\eta$-band-limited in the OLCT domain.
					Then the following reconstruction formula holds:
					\begin{eqnarray*}
						f(x)\overline{g(x-u)}e^{\frac{ia}{2b}(u-x)^2}&=&\sum_{n\in\mathbb{Z}}f(\frac{n\pi}{\eta})\overline{g(\frac{n\pi}{\eta}-u)}e^{\frac{ia}{2b}(u-\frac{n\pi}{\eta})^2}\operatorname{sinc}\big(\frac{\eta}{\pi}(x-\frac{n\pi}{\eta})\big)
					\end{eqnarray*}
				\end{theorem} 
				\begin{proof}
					It is easy to verify that if f is $\eta$ band limited signal under STOLCT, then the signal $F(x)= f(x)\overline{g(x-u)}, u\in \mathbb{R}$, is $\eta$ band limited signal under OLCT domain. 
					 \begin{eqnarray*}
					 	h(x)=e^{\frac{ia}{2b}(u-x)^2}F(x)
					 \end{eqnarray*}
					 from \eqref{sampling} we can write 
					 \begin{eqnarray*}
					 		h(x)&=&\sum_{n\in\mathbb{Z}}h\big(\frac{n\pi}{\eta}\big)\operatorname{sinc}\big(\frac{\eta}{\pi}(x-\frac{n\pi}{\eta})\big)\\
					 		e^{\frac{ia}{2b}(u-x)^2 }f(x)\overline{g(x-u)}&=&\sum_{n\in\mathbb{Z}}F(\frac{n\pi}{\eta})e^{\frac{ia}{2b}(u-\frac{n\pi}{\eta})^2}\operatorname{sinc}\big(\frac{\eta}{\pi}(x-\frac{n\pi}{\eta})\big)\\
					 		f(x)\overline{g(x-u)}&=&e^{\frac{-ia}{2b}(u-x)^2}\sum_{n\in\mathbb{Z}}f(\frac{n\pi}{\eta})\overline{g(\frac{n\pi}{\eta}-u)}e^{\frac{ia}{2b}(u-\frac{n\pi}{\eta})^2}\operatorname{sinc}\big(\frac{\eta}{\pi}(x-\frac{n\pi}{\eta})\big).
					 \end{eqnarray*}
				\end{proof}  
		\begin{example}\label{example for stolct}
		Let $f(x)=\operatorname{sinc}\!\left(\frac{\eta}{\pi}x\right)e^{i\beta x},	g(x)=e^{-x^{2}/2},$
			with bandwidth $\eta=10$, modulation frequency $\beta=2$, and window shift $u=1$. 
			The OLCT parameters are chosen as $a=b=1$. 
			For these values, the windowed signal is defined by
			$ F(x)=f(x)\overline{g(x-u)}.	$
			Although the Gaussian window is not band-limited, its rapid decay allows accurate numerical reconstruction using the sampling formula of Theorem~\eqref{sampling theorem}, thereby demonstrating the robustness of the proposed method.
			\begin{eqnarray*}
				f(x)\overline{g(x-u)}&=&e^{\frac{-ia}{2b}(u-x)^2}\sum_{n\in\mathbb{Z}}f(\frac{n\pi}{\eta})\overline{g(\frac{n\pi}{\eta}-u)}e^{\frac{ia}{2b}(u-\frac{n\pi}{\eta})^2}\operatorname{sinc}\big(\frac{\eta}{\pi}(x-\frac{n\pi}{\eta})\big)\\
				&=&e^{\frac{-ia}{2b}(u-x)^2}\sum_{n\in\mathbb{Z}}\operatorname{sinc}(n)e^{i\beta\frac{n\pi}{\eta}}e^{\frac{(\frac{n\pi}{\eta}-u)^2}{2}}e^{\frac{ia}{2b}(u-\frac{n\pi}{\eta})^2}\operatorname{sinc}\big(\frac{\eta}{\pi}(x-\frac{n\pi}{\eta})\big)
			\end{eqnarray*}
		\end{example}
		
		\begin{figure}[h!]
\centering
\begin{tabular}{cc}
\includegraphics[width=0.48\linewidth]{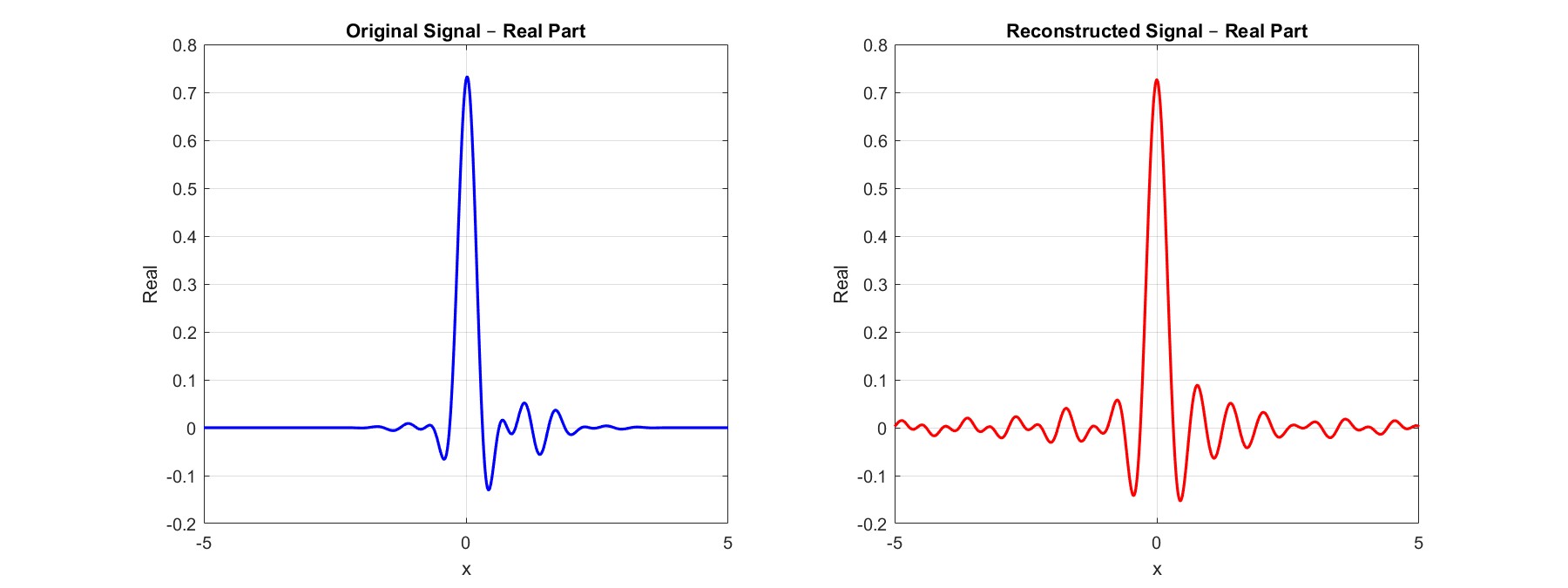} &
\includegraphics[width=0.48\linewidth]{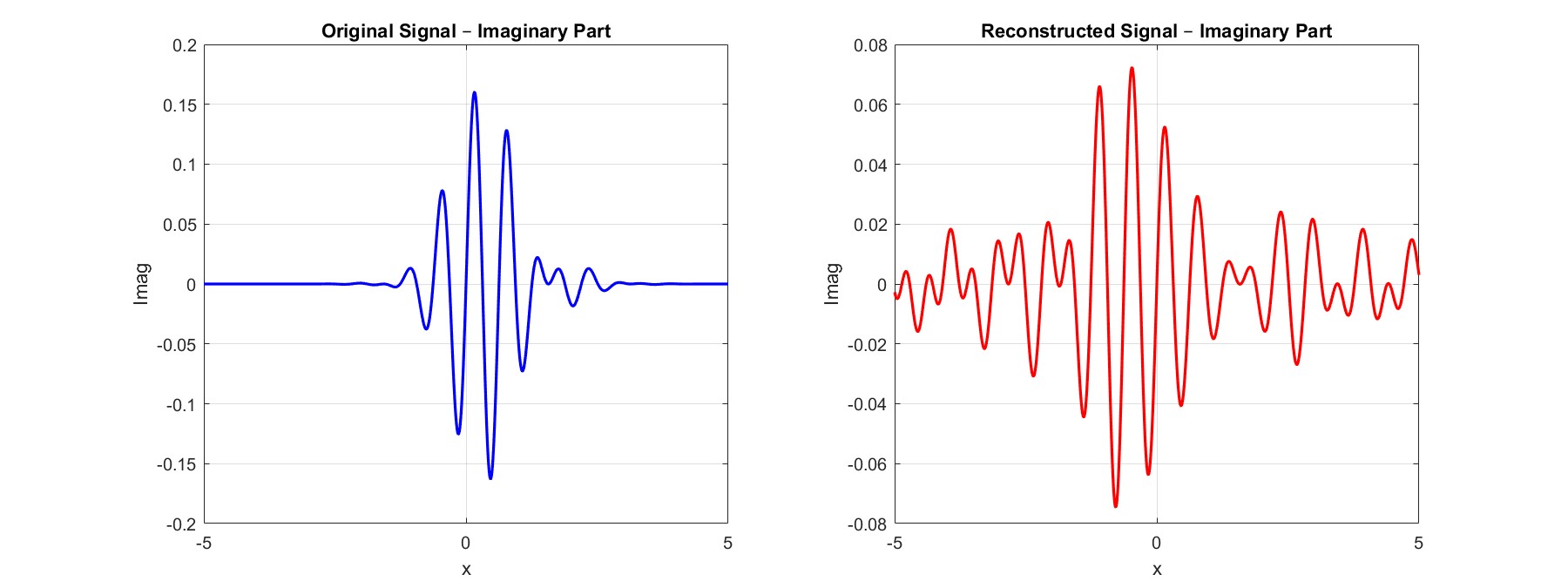} \\
(a) Real part & (b) Imaginary part \\
\end{tabular}
\caption{Real and Imaginary components for $\eta = 10$  example \ref{example for stolct}}
\label{fig:eta10_example1}
\end{figure}
	
		\begin{figure}[h!]
\centering
\includegraphics[width=1\linewidth]{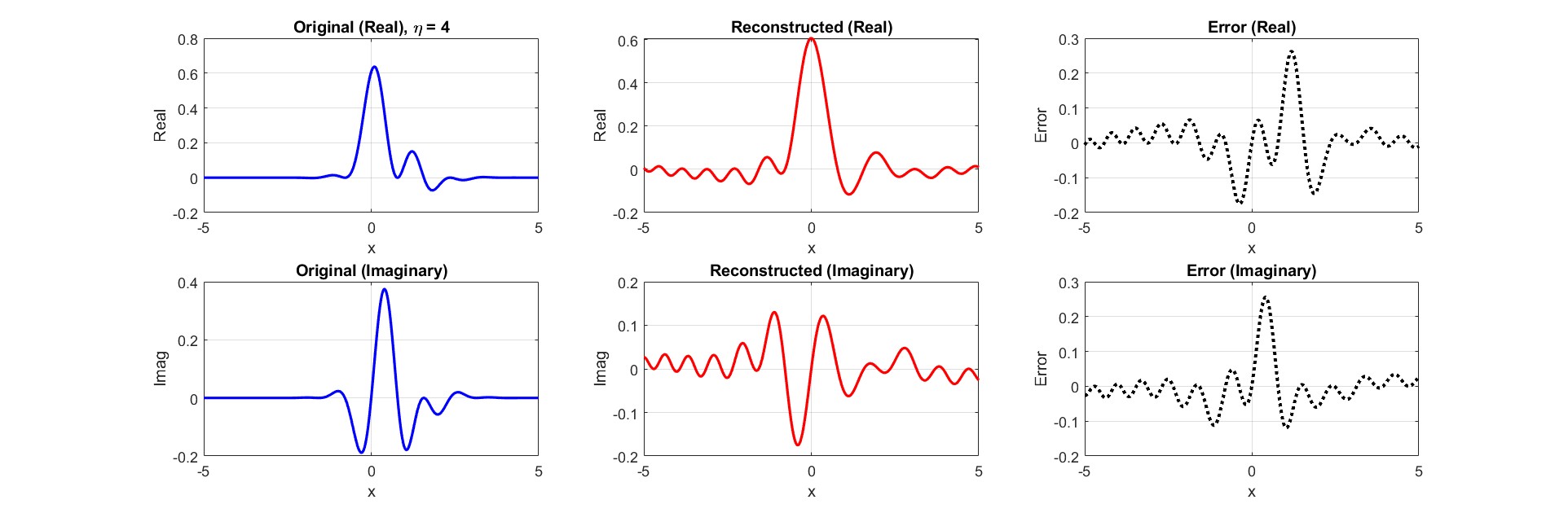}
\caption{ Real and imaginary reconstruction errors of STOLCT for $\eta = 4$}
\label{fig:error_eta4}
\end{figure}

\begin{figure}[h!]
\centering
\includegraphics[width=1\linewidth]{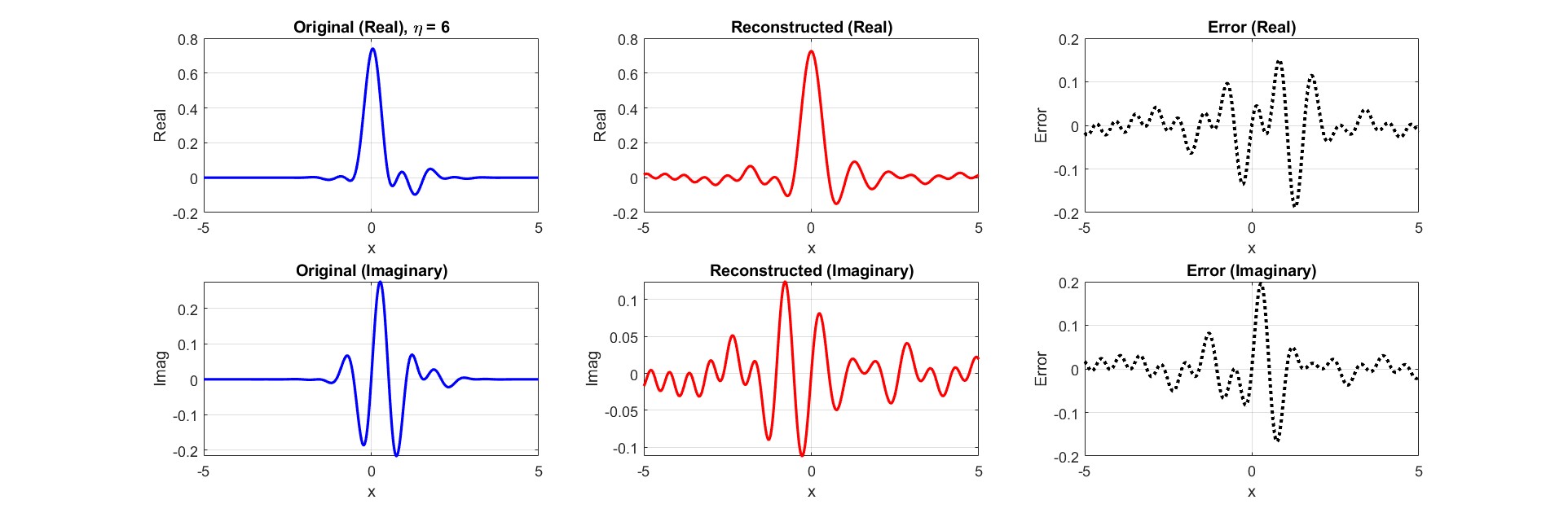}
\caption{Real and imaginary reconstruction errors of STOLCT for $\eta = 6$}
\label{fig:error_eta6}
\end{figure}

\begin{figure}[h!]
\centering
\includegraphics[width=1\linewidth]{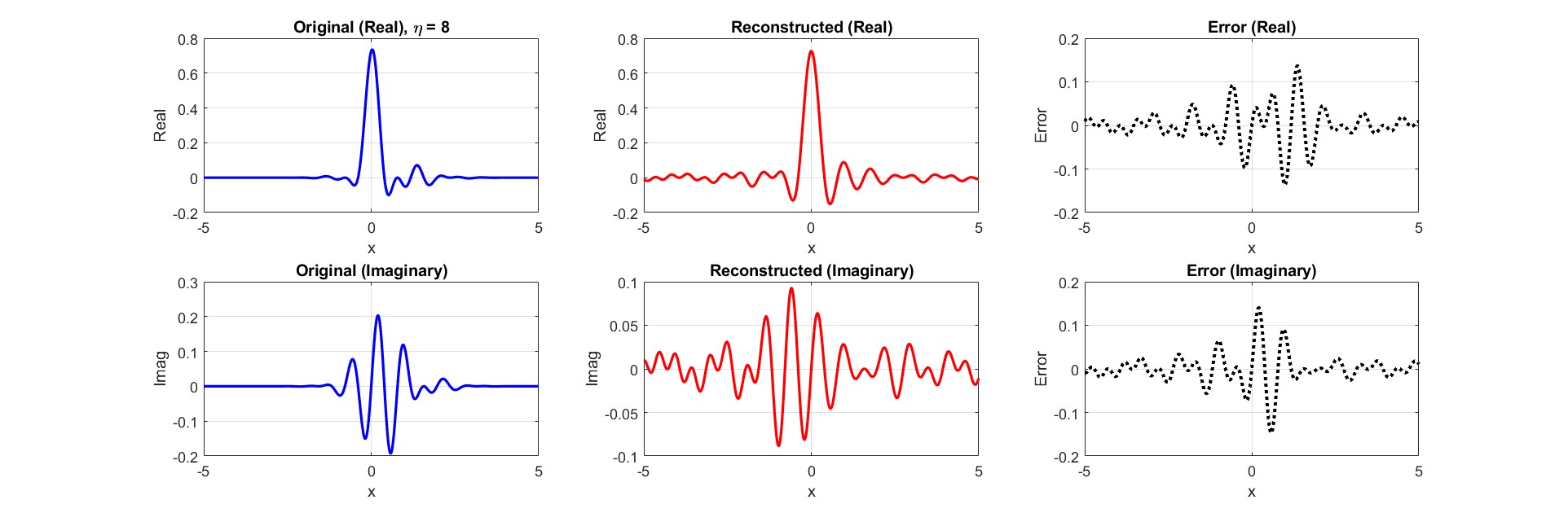}
\caption{ Real and imaginary reconstruction errors of STOLCT for $\eta = 8$}
\label{fig:error_eta8}
\end{figure}

\begin{figure}[h!]
\centering
\includegraphics[width=1\linewidth]{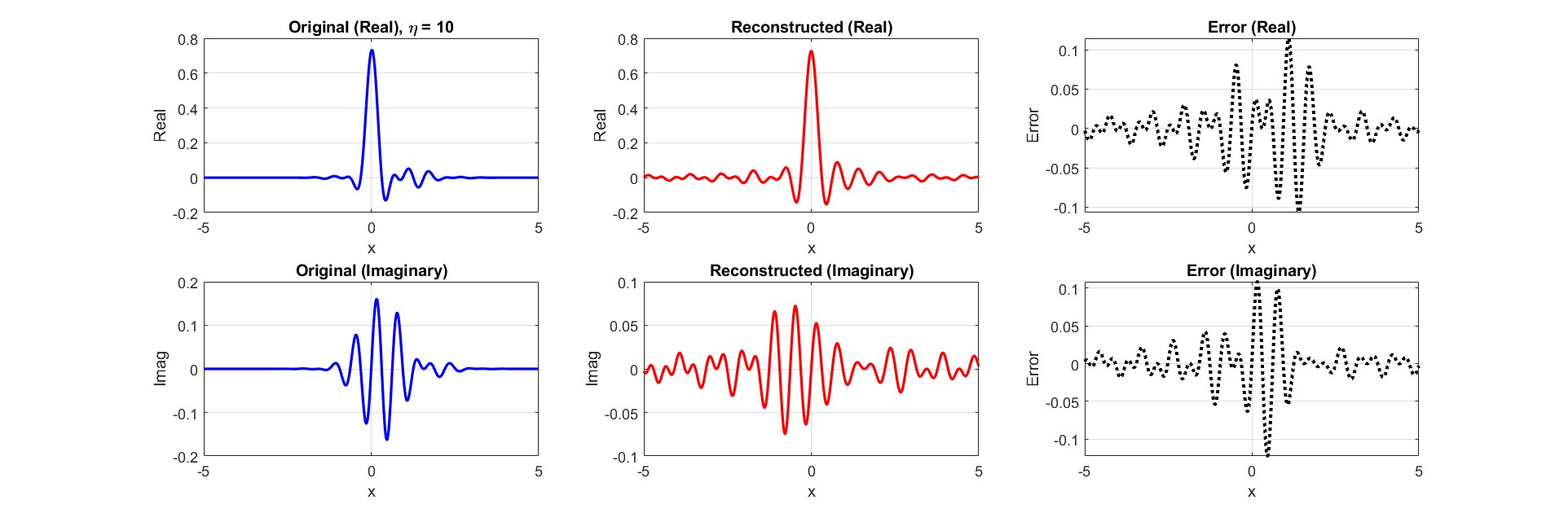}
\caption{ Real and imaginary reconstruction errors of STOLCT for $\eta = 10$}
\label{fig:error_eta10}
\end{figure}
        
        \begin{example}\label{example for stlct}
            For in STLCT we can see from\cite{A.K verma}  defined sampling for Short time LCT 
            \begin{eqnarray*}
                f(x)\overline{g(x-u)}= e^{\frac{ia}{2b}x^2}\sum_{n\in \mathbb{Z}} e^{\frac{ia}{2b}\frac{n\pi}{\eta}^2} f(\frac{n\pi}{\eta})\overline{g(\frac{n\pi}{\eta}-u)}\operatorname{sinc}(\frac{\eta x- n\pi}{b \pi})
            \end{eqnarray*}
            $
			f(x)=\operatorname{sinc}\!\left(\frac{\eta}{\pi}x\right)e^{i\beta x},	g(x)=e^{-x^{2}/2},$
			with bandwidth $\eta=10$, modulation frequency $\beta=2$, and window shift $u=1$. 
			The OLCT parameters are chosen as $a=b=1$. 
			For these values, the windowed signal is defined by
			$ F(x)=f(x)\overline{g(x-u)}.	$
             \end{example}
              \begin{table}\caption{Error Table for example \ref{example for stolct}}
			\begin{tabular}{|c|c|c|c|c|c|c|}  
				\hline                         
				$\eta$ 	&\textbf{error Real} & \textbf{error imaginary}  \\  
				\hline                    
				4& 0.25243   &   0.22734  \\  
				\hline                        
				6& 0.16711  &  0.15308 \\  
				\hline 
				8&0.12542     & 0.11478  \\
				\hline 
				10& 0.10032   &0.091791\\
				\hline
					\end{tabular}\\
                     \label{table example 2 stolct}
					
		\end{table}
             \begin{table}\caption{Error Table for example \ref{example for stlct}}
			\begin{tabular}{|c|c|c|c|c|c|c|}  
				\hline                         
				$\eta$ 	&\textbf{error Real} & \textbf{error imaginary}  \\  
				\hline                    
				4& 0.47092   &   0.27562  \\  
				\hline                        
				6& 0.40376  &  0.18364 \\  
				\hline 
				8&0.35779     & 0.13776  \\
				\hline 
				10& 0.32436   &0.11044\\
				\hline
					\end{tabular}
                    \label{table example 2 stlct}
					
		\end{table}
       
           \begin{figure}
\centering
\includegraphics[width=0.95\linewidth]{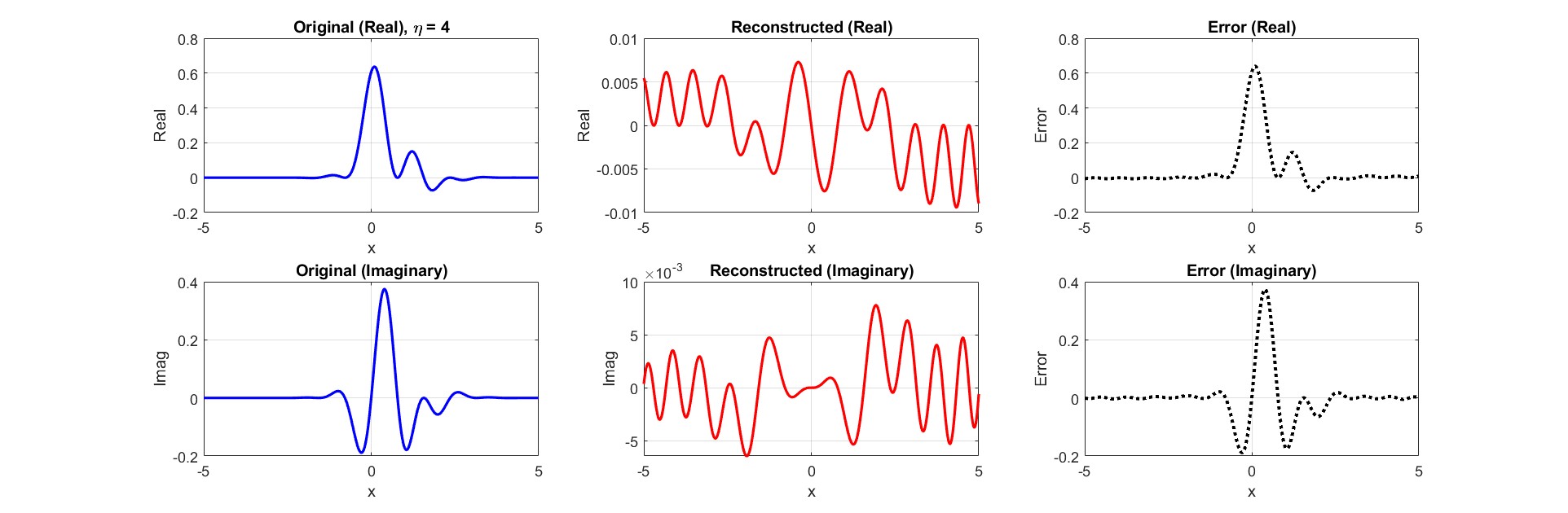}
\caption{Real and imaginary reconstruction error of STLCT for $\eta = 4$.}
\label{lct eta 4}
\end{figure}

\begin{figure}
\centering
\includegraphics[width=0.95\linewidth]{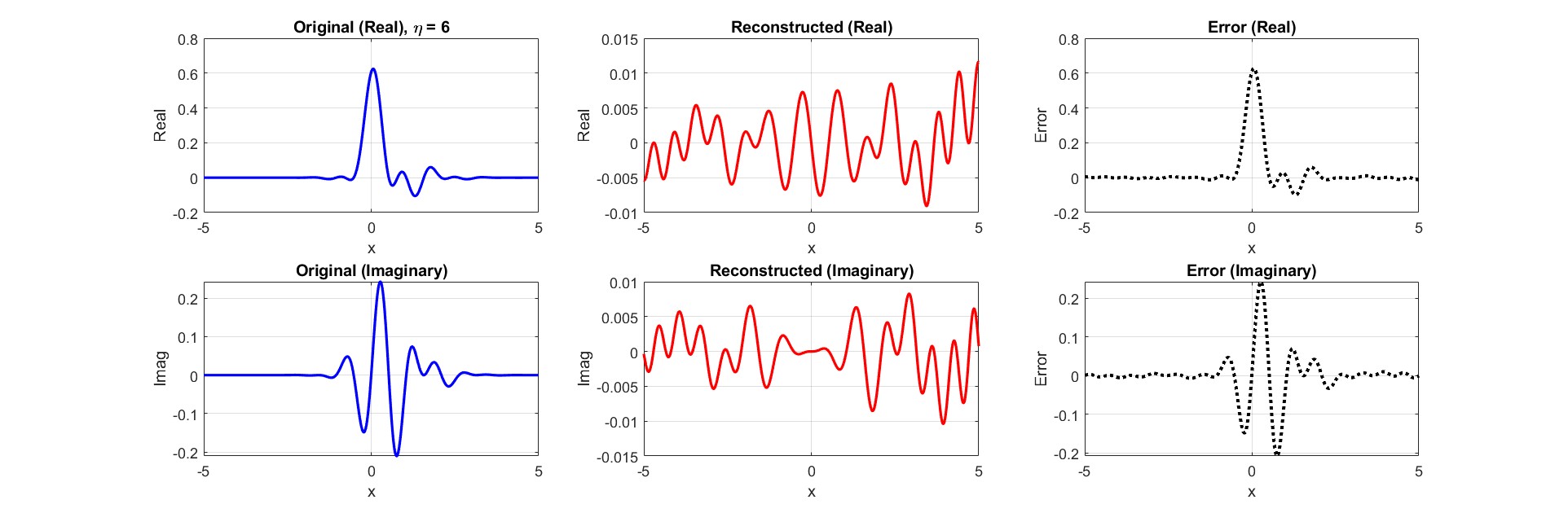}
\caption{Real and imaginary reconstruction error of STLCT for  $\eta = 6$.}
\label{lct eta 6}
\end{figure}

\begin{figure}
\centering
\includegraphics[width=0.95\linewidth]{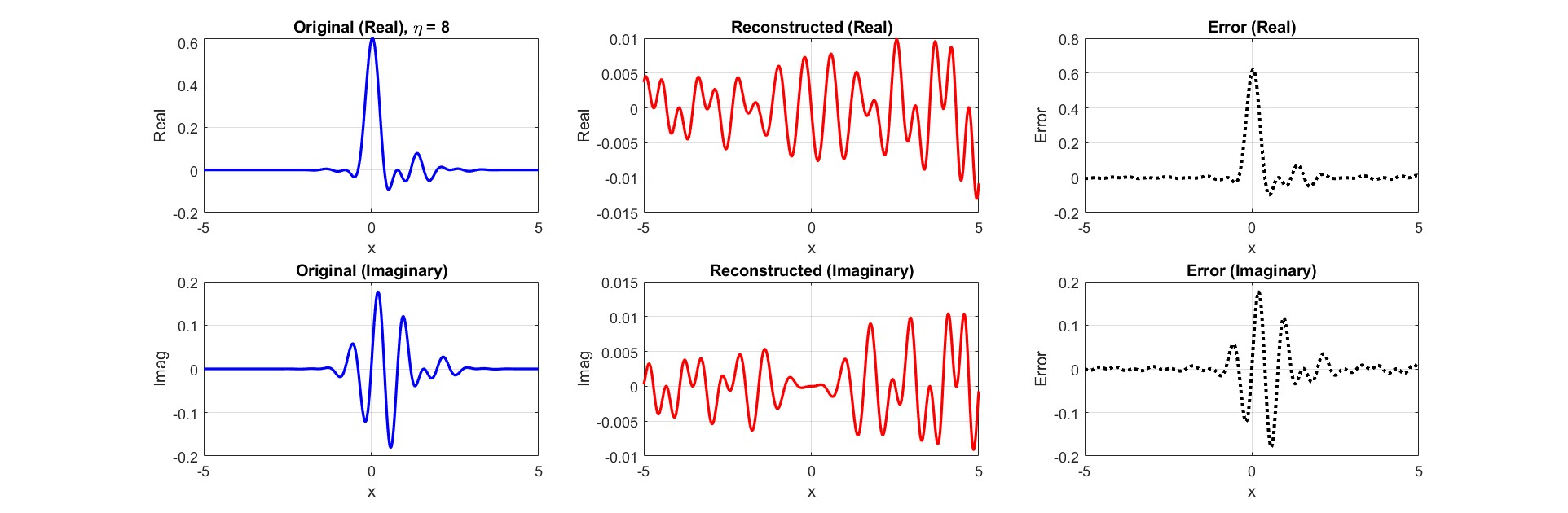}
\caption{Real and imaginary reconstruction error of STLCT for $\eta = 8$.}
\label{lct eta 8}
\end{figure}

\begin{figure}
\centering
\includegraphics[width=0.95\linewidth]{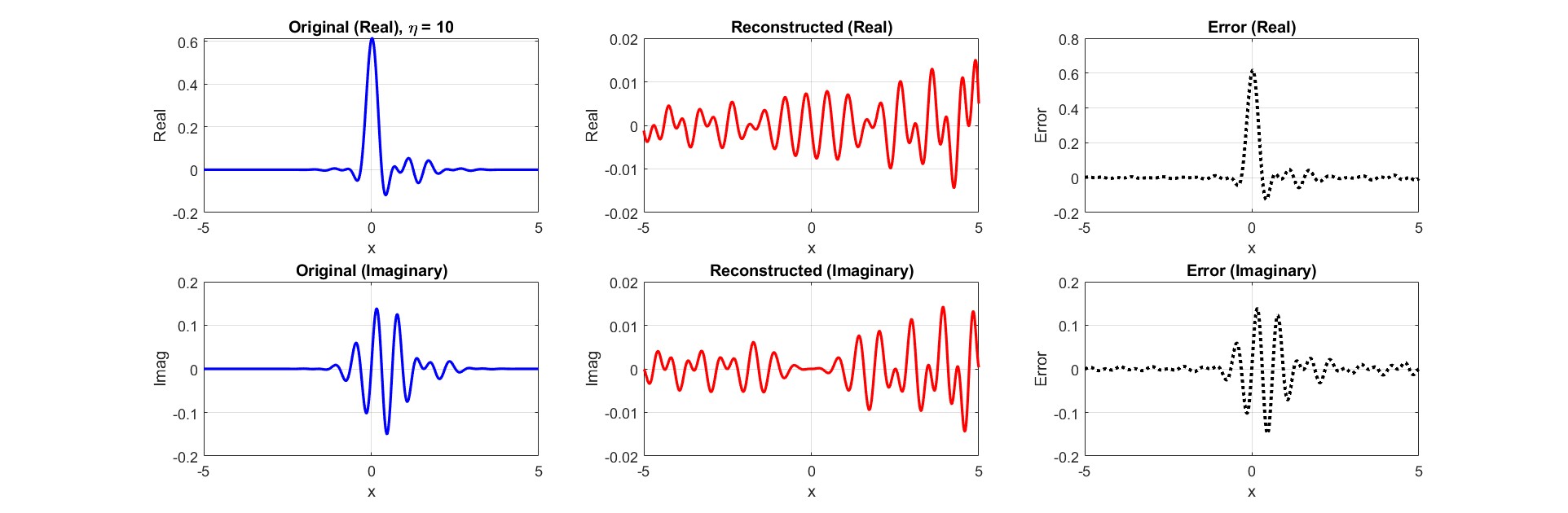}
\caption{Real and imaginary reconstruction error of STLCT for  $\eta = 10$.}
\label{lct eta 10}
\end{figure}

           \begin{remark}
         From the error graphs figure \ref{fig:error_eta4},\ref{fig:error_eta6},\ref{fig:error_eta8},\ref{fig:error_eta10},\ref{lct eta 4},\ref{lct eta 6},\ref{lct eta 8},\ref{lct eta 10}, and Tables \ref{table example 2 stolct} and \ref{table example 2 stlct} for STOLCT and STLCT, we observe that as the band-limited parameter $\eta$ increases, both the real and imaginary reconstruction errors decrease. This indicates improved accuracy and numerical stability of the reconstruction formula. Furthermore, by comparing STOLCT with STLCT, we see that STOLCT produces smaller errors, demonstrating better reconstruction performance.
          
           \end{remark}

		\FloatBarrier
		
		\begin{example}\label{example2}
			Let	$ f(x)=\operatorname{sinc}\!\left(\frac{\eta}{4\pi}x\right),
			g(x)=\operatorname{sinc}\!\left(\frac{\eta}{\pi}x\right),
			$
			with bandwidth $\eta=10$ and window shift $u=1$. 
			The OLCT parameters are chosen as $a=b=1$. 
			For these values, the windowed signal is defined by
			$
			F(x)=f(x)\overline{g(x-u)}.
			$
			Since both the signal and the window are strictly $\eta$-band-limited, all assumptions of Theorem~\eqref{sampling theorem} are satisfied. The numerical reconstruction obtained from the sampling formula coincides with $F(x)$ up to negligible truncation error, thereby providing a direct numerical validation of the proposed reconstruction formula.
		\end{example}
		 
		\begin{figure}[h!]
			\includegraphics[width=0.75\linewidth]{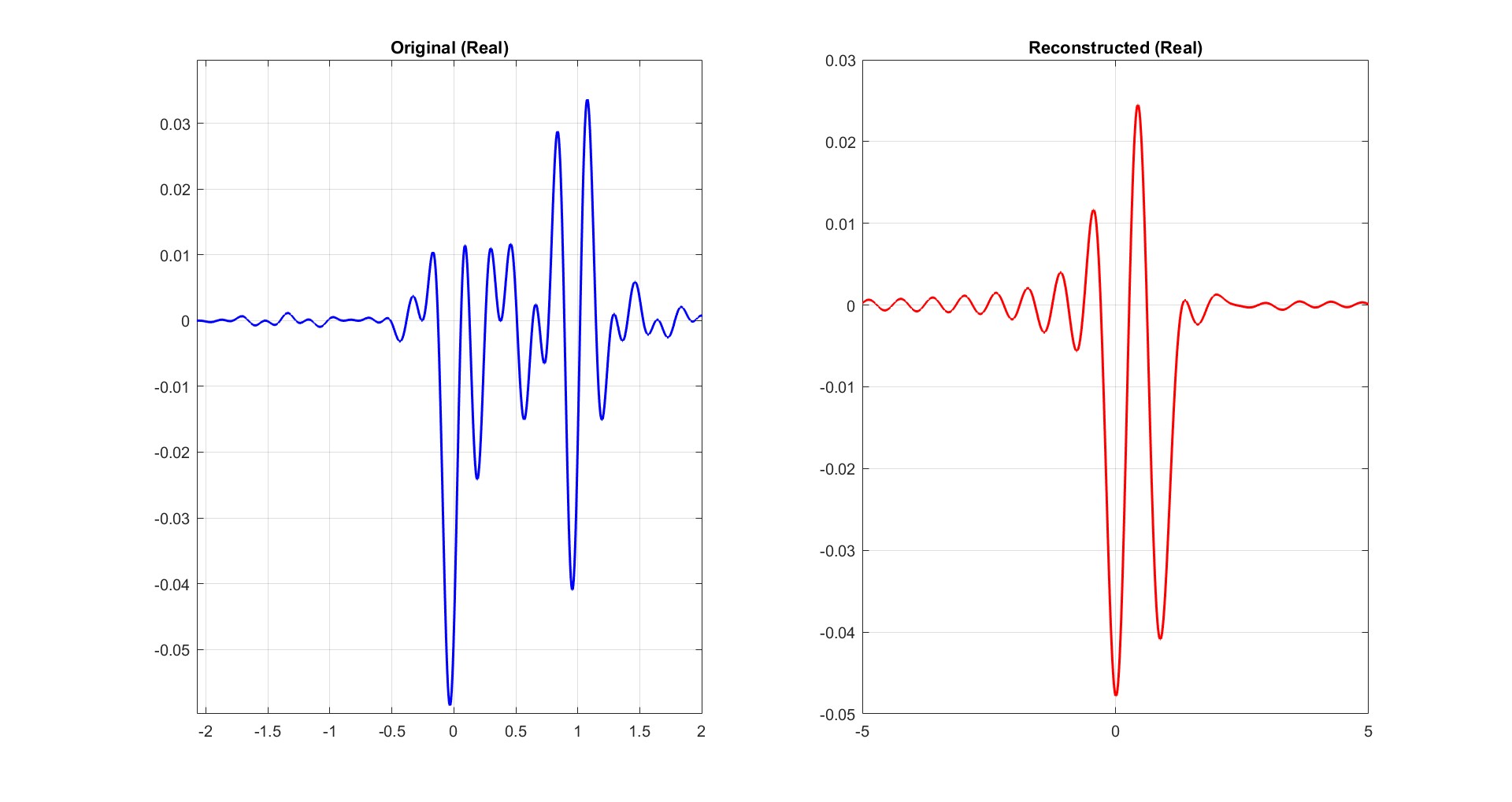}
            \caption{ Original and reconstructed for real part of signal \ref{example2} $\eta= 10$}
				\end{figure}
			\begin{figure}[h!] 
			\includegraphics[width=0.75\linewidth]{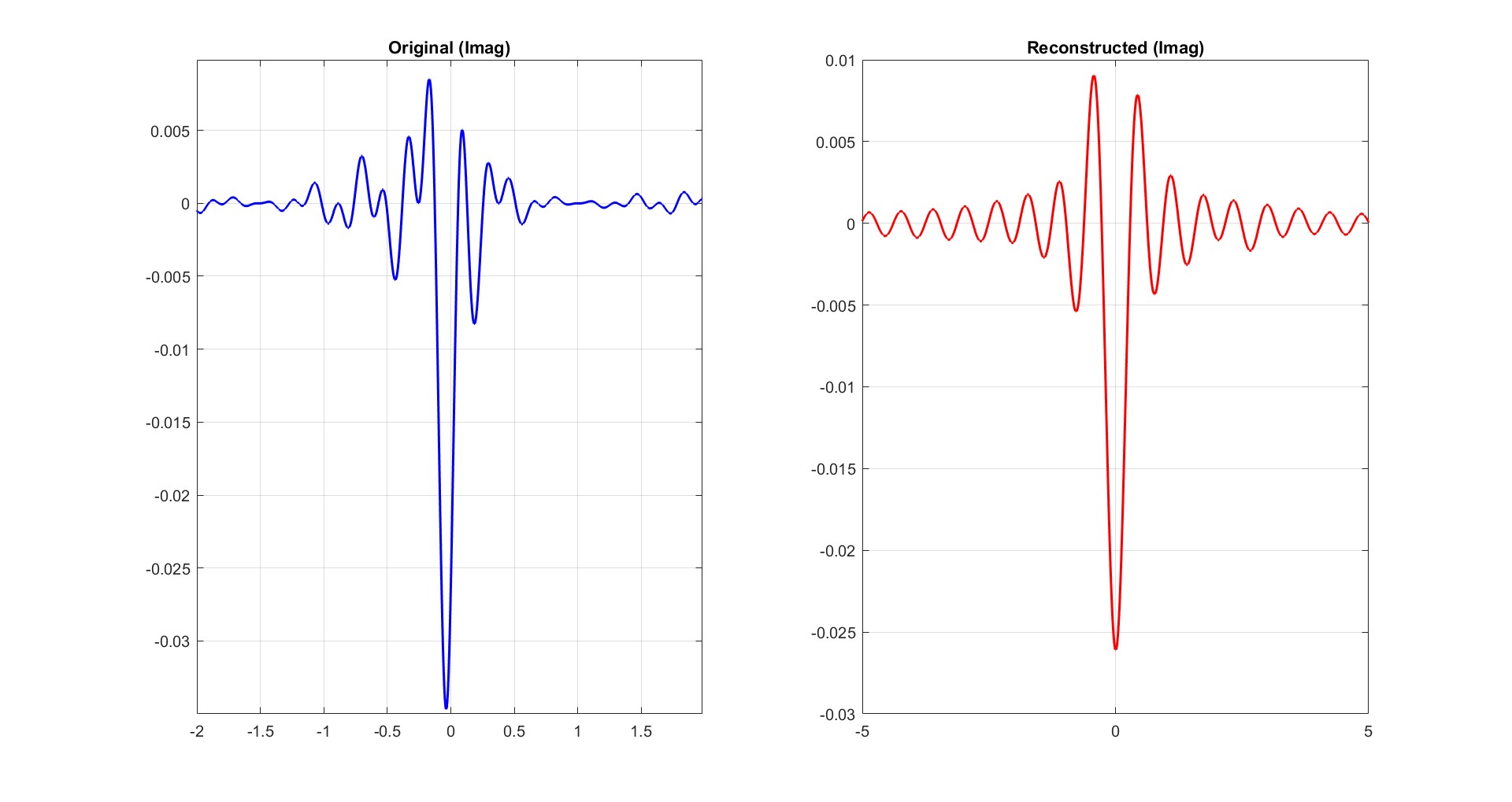}	
				\caption{ Original and reconstructed for imaginary part of signal \ref{example2} $\eta= 10$}
		\end{figure} 
        
			\begin{table}\caption{Error Table for example \ref{example2}}
			\begin{tabular}{|c|c|c|c|c|c|c|}  
				\hline                         
				$\eta$ 	&\textbf{error Real} & \textbf{error imaginary}  \\  
				\hline                    
				4& 0.015795  &   0.061774  \\  
				\hline                        
				6& 0.052966  &  0.0158849 \\  
				\hline 
				8&0.059007     & 0.027243  \\
				\hline 
				10& 0.035357   &0.012664\\
				\hline
			\end{tabular}\\
			
		\end{table}
		\begin{figure}[h!]
\centering
\includegraphics[width=0.95\linewidth]{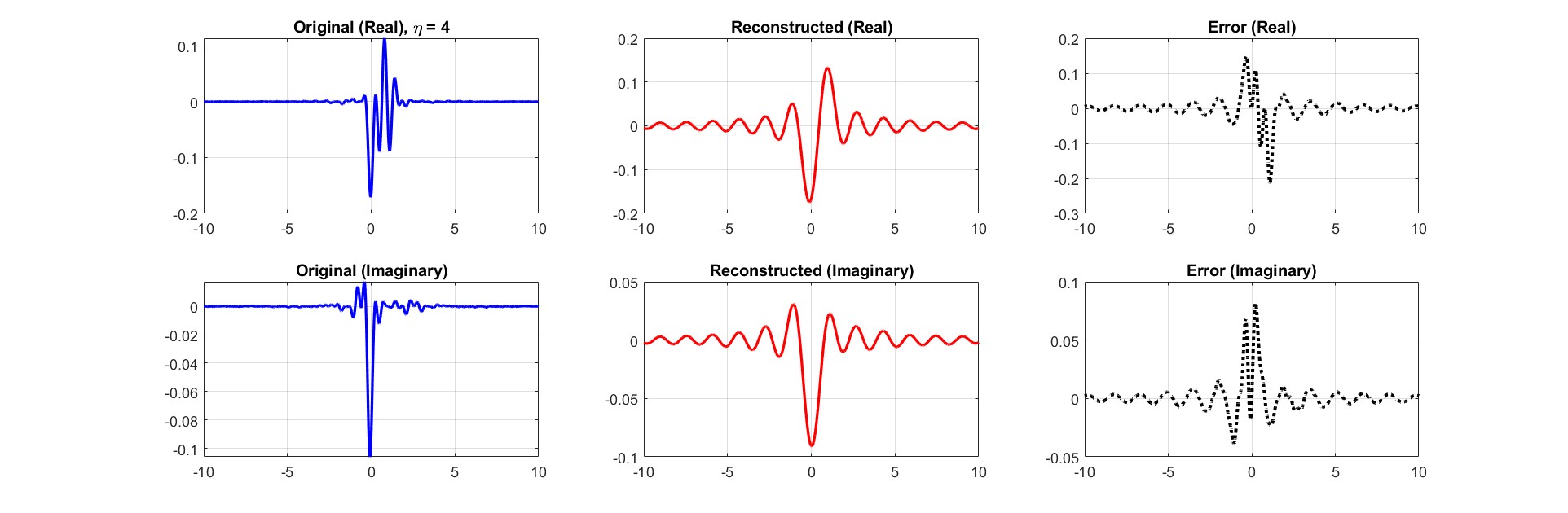}
\caption{Reconstruction error for example \ref{example2} using STOLCT $\eta = 4$.}
\label{example 2 eta 4}
\end{figure}

\begin{figure}[h!]
\centering
\includegraphics[width=0.95\linewidth]{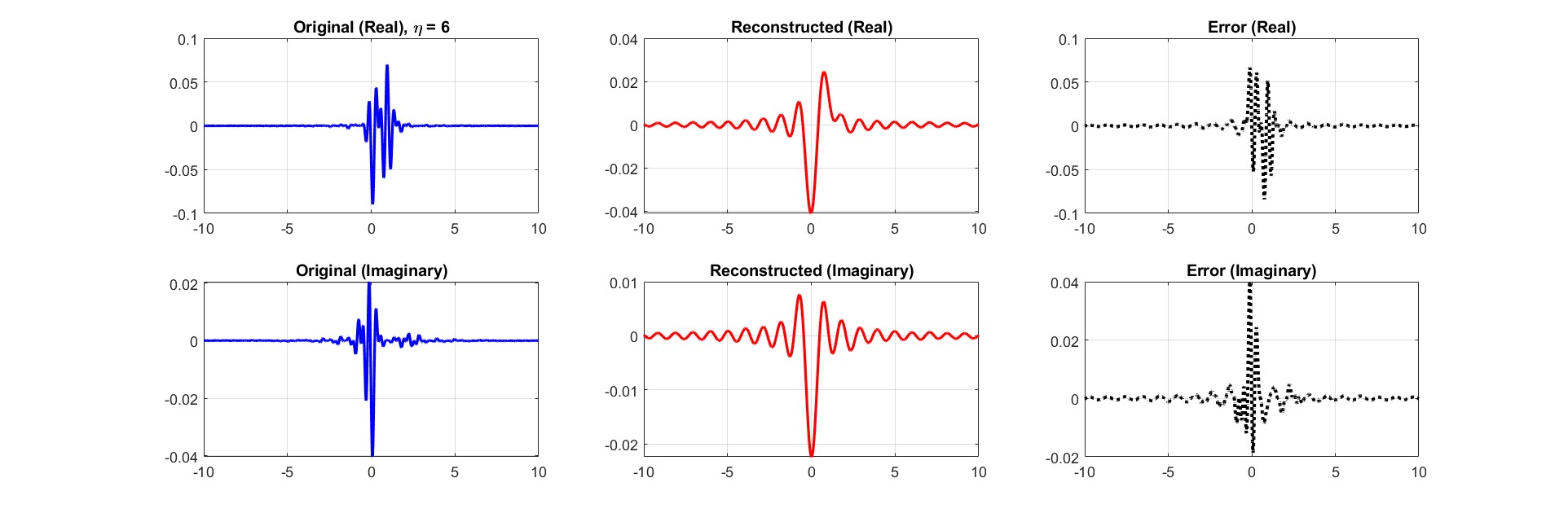}
\caption{Reconstruction error for example \ref{example2} using STOLCT $\eta = 6$.}
\label{example 2 eta 6}
\end{figure}

\begin{figure}[h!]
\centering
\includegraphics[width=0.95\linewidth]{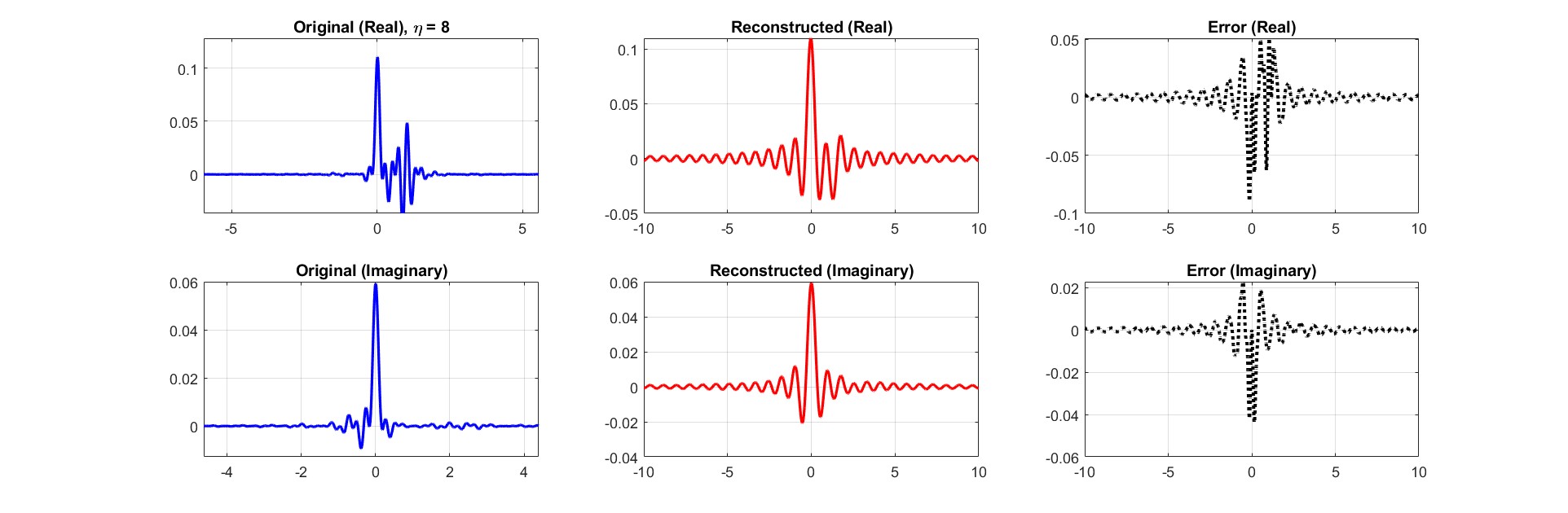}
\caption{Reconstruction error for example \ref{example2} using STOLCT $\eta = 8$.}
\label{example 2 eta 8}
\end{figure}

\begin{figure}[h!]
\centering
\includegraphics[width=0.95\linewidth]{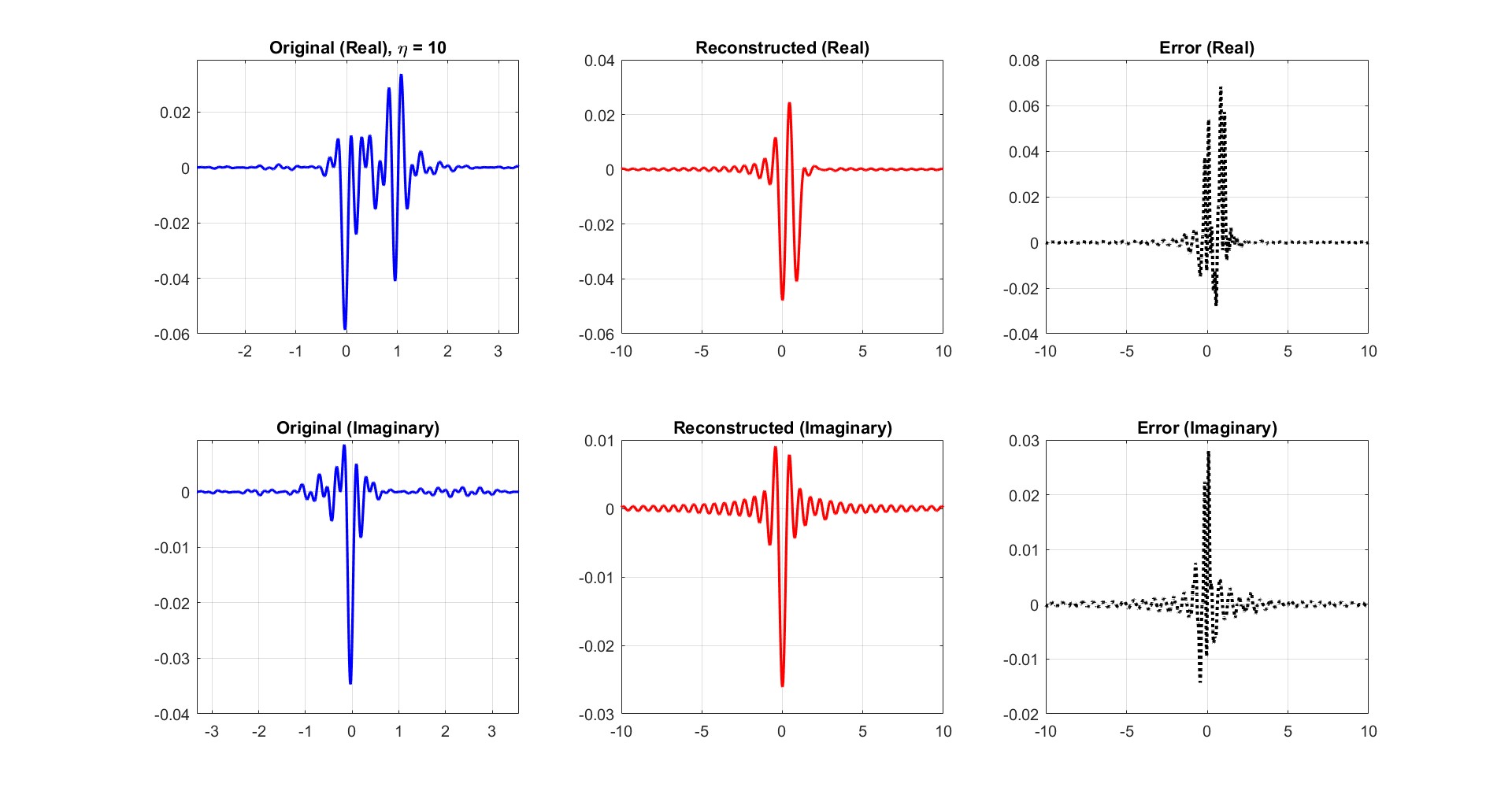}
\caption{Reconstruction error for example \ref{example2} using STOLCT $\eta = 10$.}
\label{example 2 eta 10}
\end{figure}

		\begin{remark}
		From Table 4.4 and the corresponding error graphs figure \ref{example 2 eta 4}, \ref{example 2 eta 6}, \ref{example 2 eta 8}, \ref{example 2 eta 10} it is observed that the reconstruction errors remain very small and tend to decrease as the bandwidth parameter$\eta$ increases. This behavior indicates improved approximation accuracy with higher sampling density. Since both 
		$f(x)$ and $g(x)$ are strictly $\eta$-band-limited, the reconstruction is theoretically exact, and the remaining discrepancies are solely due to numerical truncation. The results therefore demonstrate the convergence, stability, and effectiveness of the proposed STOLCT-based sampling method.
		\end{remark}

			\section{Conclusion}

           In this paper, we developed the STOLCT using a convolution-based approach. We established its inversion formula, range characterization theorem, and orthogonality relations. As applications, we discussed the Poisson summation formula, Paley–Wiener criterion, and sampling theorem. Numerical simulations and graphical error analysis demonstrated that the convolution-based OLCT provided lower reconstruction error and better performance than WLCT, owing to its offset (time-shift and frequency-shift) parameters.
	\section*{Acknowledgment}
	The research facilities provided by SRM University-AP, Amaravati are acknowledge with thanks from the authors.
	\section*{Funding statement} 
	This research received no specific grant from any funding agency in the public, commercial, or not-for-profit sectors.
	\section*{Declaration of competing interest}
	The authors declare that they have no known competing financial interests or personal relationships that could have appeared to influence the work reported in this paper.
	\\\\
	\bibliographystyle{amsplain}
		
\end{document}